\documentclass[12pt]{amsart}
\usepackage{amssymb}
\usepackage{amsmath}
\usepackage{mathrsfs}
\usepackage{enumerate}
\usepackage{amsthm}
\usepackage{cancel}
\newtheorem{theorem}{Theorem}[section]
\newtheorem{lemma}[theorem]{Lemma}
\newtheorem{prop}[theorem]{Proposition}
\newtheorem{corollary}[theorem]{Corollary}
\theoremstyle{definition}
\newtheorem{definition}[theorem]{Definition}
\newtheorem{example}[theorem]{Example}
\newtheorem{remark}[theorem]{Remark}
\newtheorem{problem}{Problem}[]

\numberwithin{equation}{section}

\newcommand{\opnorm}[1]{\|{#1}\|_\mathrm{op}}
\newcommand{\fnorm}[1]{\|{#1}\|_\mathrm{F}}

\newcommand{\lnorm}[2]{\|{#1}\|_{#2}}
\newcommand{\inner}[2]{\langle{#1, #2}\rangle}

\begin{document}
\allowdisplaybreaks
\title{Testing the mixture model hypothesis via spectral gap}
\author{March T.~Boedihardjo, Joe Kileel, and Vandy Tombs}
\begin{abstract}
In this paper, we study the problem of testing whether or not a given probability measure $\mu$ on $\mathbb{R}^{d}$ can be decomposed as a mixture of two probability measures whose second order statistics are significantly different. 
We call this the problem of testing the mixture model hypothesis.
To tackle it, we introduce a new set of computable orthogonal invariants of $\mu$, namely, the eigenvalues of the 4th moment operator $T_{\mu}$ associated with the measure. 
We prove that the largest eigenvalue is always an outlier eigenvalue.  Further, we show how the first and second largest eigenvalues of $T_{\mu}$ give nonasymptotic bounds for this problem and give a complete resolution of the asymptotic version of the problem under the $L^{8}$-$L^{2}$ equivalence assumption.
\end{abstract}
\maketitle
\section{Introduction}
\subsection{Formulating the mixture model hypothesis}
Let $d\in\mathbb{N}$. For $p\geq 1$, denote by $\mathcal{P}_{p}(\mathbb{R}^{d})$ the set of all probability measures on $\mathbb{R}^{d}$ such that $\int_{\mathbb{R}^{d}}\|x\|_{2}^{p}\,d\mu(x)<\infty$, where $\|\,\|_{2}$ is the Euclidean norm on $\mathbb{R}^{d}$. If $x\in\mathbb{R}^{d}$, then $\delta_{x}$ denotes the probability measure with an atom at $x$ of mass $1$.

Suppose that we are given a probability measure $\mu\in\mathcal{P}_{2}(\mathbb{R}^{d})$. Our goal is to test whether $\mu$ is a mixture of two very different distributions, or equivalently, whether the population represented by $\mu$ is, in fact, the union of two very different subpopulations. 
We call this the problem of testing the mixture model hypothesis.
There are several known approaches to this problem:
\begin{enumerate}[(1)]
\item (Multimodal/unimodal) When the dimension $d$ is small, a common approach is to check whether $\mu$ is a multimodal distribution or a unimodal distribution, since a distribution being multimodal suggests that it is derived from a mixture model. However, when the dimension $d$ gets large, oftentimes  the mode of $\mu$ either does not make sense or gives very little information.
For example, an empirical measure on a high dimensional space is often the uniform distribution on $n$ points that are well separated, where $n$ is the sample size; in this case, every one of those $n$ points is a mode of $\mu$. Therefore, checking the multimodality/unimodality of $\mu$ is typically not a good approach for large $d$.
\item (Parametric model) 
Suppose we fix a parametric class $\mathcal{C}$ of probability distributions, and  assume that either $\mu$ comes from $\mathcal{C}$ or is a mixture of two distributions in $\mathcal{C}$.
Then the mixture model hypothesis can be decided as a byproduct of parameter estimation.  
Several parameter estimation algorithms have been developed for particular classes of parametric mixture models, including mixtures of two Gaussian distributions \cite{kalai2010efficiently}.  Generally, the costs scale inverse polynomially in the separation between the two mixture constituents. Of course, this approach is also reliant on the assumption of the parameteric model $\mathcal{C}$. 
\item (Cheeger/Poincar\'e constant) The Cheeger constant $\psi_{\mu}$ and Poincar\'e constant $C_{p}(\mu)$ (see \cite{KL}) quantify the ``metric disconnectedness" of a given probability measure $\mu$, or more precisely, the extent to which the population represented by $\mu$ can be partitioned into two subpopulations that are mostly separated like an hourglass. However, in practice, one often uses an empirical measure of $\mu$ to approximate $\mu$.  Unfortunately, the quantities $\psi_{\mu}$ and $C_{P}(\mu)$ are always equal to $0$ or $\infty$ when $\mu$ is a discrete measure.
\end{enumerate}
Therefore, in the context of high dimensions and in a  nonparametric setting, it is not clear what a good approach to testing the mixture model hypothesis is.  Moreover, it is not immediate what even  a useful and precise formulation of the problem should be.


In this paper, we 
choose to formulate the problem as testing whether a given $\mu\in\mathcal{P}_{2}(\mathbb{R}^{d})$ can be decomposed as a mixture $\mu=\frac{1}{2}\mu_{1}+\frac{1}{2}\mu_{2}$ of two probability measures $\mu_{1},\mu_{2}$ on $\mathbb{R}^{d}$ whose second order statistics are significantly different.
To quantify the extent to which $\mu\in\mathcal{P}_{2}(\mathbb{R}^{d})$ has such a mixture decomposition, we use the following notion of second order separation parameter.
\begin{definition}\label{defsmu}
Let $\mu\in\mathcal{P}_{2}(\mathbb{R}^{d})$. The {\it second order statistics matrix} of $\mu$ is defined to be the $d\times d$ positive semidefinite matrix:
\[B_{\mu}:=\int_{\mathbb{R}^{d}}xx^{T}\,d\mu(x),\]
where $x^{T}$ is the transpose of $x$. The {\it second order separation parameter} of $\mu$ is defined by
\begin{equation}\label{defsmueq}
s(\mu):=\frac{1}{2}\cdot\sup_{\mu=\frac{1}{2}\mu_{1}+\frac{1}{2}\mu_{2}}\|B_{\mu_{1}}-B_{\mu_{2}}\|_{\mathrm{F}},
\end{equation}
where the supremum is over all probability measures $\mu_{1},\mu_{2}$ on $\mathbb{R}^{d}$ such that $\mu=\frac{1}{2}\mu_{1}+\frac{1}{2}\mu_{2}$, and $\|\,\|_{\mathrm{F}}$ denotes the Frobenius norm.
\end{definition}

\begin{remark} 
In (\ref{defsmueq}), since $B_{\mu_{1}}$ and $B_{\mu_{2}}$ are positive semidefinite, we have
\begin{equation} \label{eq:s_remark}
s(\mu) = \frac{1}{2}\cdot\sup_{\mu=\frac{1}{2}\mu_{1}+\frac{1}{2}\mu_{2}}\|B_{\mu_{1}}-B_{\mu_{2}}\|_{\mathrm{F}}\leq\frac{1}{2}\cdot\sup_{\mu=\frac{1}{2}\mu_{1}+\frac{1}{2}\mu_{2}}\|B_{\mu_{1}}+B_{\mu_{2}}\|_{\mathrm{F}}=\|B_{\mu}\|_{\mathrm{F}},
\end{equation}
where the last equality follows from the fact that if $\mu=\frac{1}{2}\mu_{1}+\frac{1}{2}\mu_{2}$, then $B_{\mu}=\frac{1}{2}B_{\mu_{1}}+\frac{1}{2}B_{\mu_{2}}$.
\end{remark}

\begin{remark}
The supremum in (\ref{defsmueq}) is taken over all possible equal weight mixture decompositions $\mu=\frac{1}{2}\mu_{1}+\frac{1}{2}\mu_{2}$. For unequal weight mixtures, see Section \ref{unequalsection}. That section also shows how to tweak $\mu$ so that we can also take into account the first order statistics of $\mu_{1}$ and $\mu_{2}$ inside the supremum in (\ref{defsmueq})
\end{remark}

In this paper, we formalize testing the mixture model hypothesis as follows. 
\begin{problem}\label{mainproblem}
Given a probability measure $\mu\in\mathcal{P}_{2}(\mathbb{R}^{d})$ and a threshold $0<\epsilon<1$, determine which category of the following dichotomy $\mu$ belongs to:
\begin{center}
Mixture model hypothesis: $s(\mu)>\epsilon\|B_{\mu}\|_{\mathrm{F}}$\\
Single distribution hypothesis: $s(\mu)\leq\epsilon\|B_{\mu}\|_{\mathrm{F}}$.
\end{center}
\end{problem}
\begin{remark}
The mixture model hypothesis in Problem \ref{mainproblem} is equivalent to saying that there exists a mixture decomposition $\mu=\frac{1}{2}\mu_{1}+\frac{1}{2}\mu_{2}$ such that $\|B_{\mu_{1}}-B_{\mu_{2}}\|_{\mathrm{F}}>\epsilon\|B_{\mu}\|_{\mathrm{F}}$, i.e., $B_{\mu_{1}}$ and $B_{\mu_{2}}$ differ by as much as $B_{\mu}$ in the Frobenius norm. The single distribution hypothesis is equivalent to saying that no matter how we decompose $\mu=\frac{1}{2}\mu_{1}+\frac{1}{2}\mu_{2}$ as a mixture, we must have $\|B_{\mu_{1}}-B_{\mu_{2}}\|_{\mathrm{F}}\leq\epsilon\|B_{\mu}\|_{\mathrm{F}}$, i.e., $B_{\mu_{1}}$ and $B_{\mu_{2}}$ must differ little relative to $B_{\mu}$ in the Frobenius norm.
\end{remark}

We also study the following problem, which is an asymptotic version of Problem \ref{mainproblem}. For a given sequence $(\mu^{(n)})_{n\in\mathbb{N}}$ of probability measures, we ask whether it is true that for every threshold $0<\epsilon<1$, the sequence eventually satisfies the single distribution hypothesis $s(\mu^{(n)})\leq\epsilon\|B_{\mu^{(n)}}\|_{\mathrm{F}}$. This condition is equivalent to $\displaystyle\lim_{n\to\infty}\frac{s(\mu^{(n)})}{\|B_{\mu^{(n)}}\|_{\mathrm{F}}}=0$.

\begin{problem}\label{mainproblemasymptotic}
Given a sequence $(\mu^{(n)})_{n\in\mathbb{N}}$ where each $\mu^{(n)}$ is a probability measure on $\mathbb{R}^{d_{n}}$ with $d_{n}\in\mathbb{N}$ and $\mu^{(n)}\neq\delta_{0}$, determine whether or not $\displaystyle\lim_{n\to\infty}\frac{s(\mu^{(n)})}{\|B_{\mu^{(n)}}\|_{\mathrm{F}}}=0$.
\end{problem}

\begin{remark}
In most cases, in order for $\displaystyle\lim_{n\to\infty}\frac{s(\mu^{(n)})}{\|B_{\mu^{(n)}}\|_{\mathrm{F}}}=0$ to occur, the dimension $d_{n}$ needs to grow to $\infty$ as $n\to\infty$. 
This is because if $d_{n}=d$ is fixed and all the measures $\mu^{(n)}$ are supported on a compact set (say the unit ball of $\mathbb{R}^{d}$),  there is a subsequence $(\mu^{(n_{k})})_{k\in\mathbb{N}}$ that converges to a probability measure $\mu$ on $\mathbb{R}^{d}$. 
The condition in Problem \ref{mainproblemasymptotic} would then require $xx^{T}$ to be a constant matrix $\mu$-almost surely.
So in particular, Problem \ref{mainproblemasymptotic} is generally not suitable for testing a sequence of empirical measures of a fixed measure $\mu$ on $\mathbb{R}^{d}$.
\end{remark}

We now go through some examples that illustrate the two hypotheses in Problems \ref{mainproblem} and \ref{mainproblemasymptotic}. The bounds for $s(\mu)$ stated in these examples are proved later in this paper. If $\Sigma$ is a $d\times d$ positive semidefinite matrix, then $\mathcal{N}(0,\Sigma)$ denotes the normal distribution on $\mathbb{R}^{d}$ with mean $0$ and covariance $\Sigma$.

\begin{example}[Product measure]\label{introproduct}
Suppose that $X_{1},\ldots,X_{d}$ are independent, identically distributed random variables with $\mathbb{E}X_{1}=0$, $\mathbb{E}X_{1}^{2}=1$ and $\mathbb{E}X_{1}^{4}<\infty$. Let $\mu^{(d)}$ be the distribution of the random vector $X=(X_{1},\ldots,X_{d})^{T}$ in $\mathbb{R}^{d}$. Then $B_{\mu^{(d)}}=I$, so $\|B_{\mu^{(d)}}\|_{\mathrm{F}}=\sqrt{d}$. On the other hand, as we prove later in Section \ref{smuexamplesproofs}, we have $s(\mu^{(d)})\leq(\mathbb{E}X_{1}^{4}+1)^{1/2}$, regardless of the dimension $d$. Therefore, concerning Problem \ref{mainproblemasymptotic}, we have that the sequence $(\mu^{(d)})_{d\in\mathbb{N}}$ satisfies
\[\lim_{d\to\infty}\frac{s(\mu^{(d)})}{\|B_{\mu^{(d)}}\|_{\mathrm{F}}}=0.\]
Equivalently, for any fixed threshold $0<\epsilon<1$, as the dimension $d$ gets large, we have $s(\mu)\leq\epsilon\|B_{\mu}\|_{\mathrm{F}}$, i.e., $\mu$ satisfies the single distribution hypothesis in Problem \ref{mainproblem}.

As a special case, when $X_{1}$ is a Rademacher random variable, $\mu^{(d)}$ is the uniform distribution on the discrete hypercube $\{-1,1\}^{d}$.  Also, when $X_{1}$ is a standard normal random variable, $\mu^{(d)}$ is the standard normal distribution on $\mathbb{R}^{d}$. The example shows that for large dimensions $d$, the uniform distribution on $\{-1,1\}^{d}$ and the standard normal distribution $\mathcal{N}(0,I)$ both satisfy the single distribution hypothesis in Problem \nolinebreak \ref{mainproblem}. 
\end{example}

\begin{example}[Single normal]\label{introsinglenormal}
Suppose that $\mu=\mathcal{N}(0,\Sigma)$, where $\Sigma$ is a $d\times d$ positive semidefinite matrix. Then $B_{\mu}=\Sigma$, so $\|B_{\mu}\|_{\mathrm{F}}=\|\Sigma\|_{\mathrm{F}}$. As we prove later in Section \ref{smuexamplesproofs},
\[0.8\cdot\|\Sigma\|_{\mathrm{op}}\leq s(\mu)\leq \sqrt{2}\cdot\|\Sigma\|_{\mathrm{op}},\]
where $\|\,\|_{\mathrm{op}}$ denotes the spectral norm. Therefore, for a fixed threshold $0<\epsilon<1$, if the stable rank $(\frac{\|\Sigma\|_{\mathrm{F}}}{\|\Sigma\|_{\mathrm{op}}})^{2}$ is at least $\frac{2}{\epsilon^{2}}$, then we have $\sqrt{2}\cdot\|\Sigma\|_{\mathrm{op}}\leq\epsilon\|\Sigma\|_{\mathrm{F}}$, and so $s(\mu)\leq\epsilon\|B_{\mu}\|_{\mathrm{F}}$. This implies that $\mu$ satisfies the single distribution hypothesis in Problem \ref{mainproblem}.
\end{example}

\begin{example}[Mixture of two orthogonal subspaces]\label{intromixture}
Suppose that $\mu=\frac{1}{2}\mu_{1}+\frac{1}{2}\mu_{2}$ is the equal weight mixture of a probability measure $\mu_{1}$ supported on a subspace $V$ of $\mathbb{R}^{d}$ and a probability measure $\mu_{2}$ supported on the orthogonal complement $V^{\perp}$. Then $B_{\mu}=\frac{1}{2}B_{\mu_{1}}+\frac{1}{2}B_{\mu_{2}}$, so $\|B_{\mu}\|_{\mathrm{F}}=\frac{1}{2}\left\|B_{\mu_{1}}+B_{\mu_{2}}\right\|_{\mathrm{F}}$. By the definition of $s(\mu)$, we have $s(\mu)\geq\frac{1}{2}\|B_{\mu_{1}}-B_{\mu_{2}}\|_{\mathrm{F}}$. Since $\mu_{1}$ and $\mu_{2}$ are supported on orthogonal subspaces, the matrices $B_{\mu_{1}}$ and $B_{\mu_{2}}$ when treated as vectors are orthogonal, i.e., the trace of $B_{\mu_{1}}B_{\mu_{2}}$ is $0$.  So,
\[\|B_{\mu_{1}}-B_{\mu_{2}}\|_{\mathrm{F}}=(\|B_{\mu_{1}}\|_{\mathrm{F}}^{2}+\|B_{\mu_{2}}\|_{\mathrm{F}}^{2})^{1/2}=\|B_{\mu_{1}}+B_{\mu_{2}}\|_{\mathrm{F}}.\]
Therefore, $s(\mu)\geq\|B_{\mu}\|_{\mathrm{F}}$. But by (\ref{eq:s_remark}), we also have $s(\mu)\leq\|B_{\mu}\|_{\mathrm{F}}$. Hence, $s(\mu)=\|B_{\mu}\|_{\mathrm{F}}$, and so $\mu$ satisfies the mixture model hypothesis in Problem \ref{mainproblem} as long as $\mu\neq\delta_{0}$.
\end{example}

\begin{example}[Gaussian mixture]\label{introgeneralmixture}
Suppose that $\Sigma_{1}$ and $\Sigma_{2}$ are $d\times d$ positive semidefinite matrices. Let
\[\mu=\frac{1}{2}\mathcal{N}(0,\Sigma_{1})+\frac{1}{2}\mathcal{N}(0,\Sigma_{2})\]
be the equal weight mixture of $\mathcal{N}(0,\Sigma_{1})$ and $\mathcal{N}(0,\Sigma_{2})$. Depending on whether the two covariances $\Sigma_{1}$ and $\Sigma_{2}$ differ significantly, the measure $\mu$ may or may not satisfy the mixture model hypothesis in Problem \ref{mainproblem}. We give here tight upper and lower bounds for $B_{\mu}$ and $s(\mu)$. Since $B_{\mu}=\frac{1}{2}\Sigma_{1}+\frac{1}{2}\Sigma_{2}$, we have
\[\frac{1}{2}\max(\|\Sigma_{1}\|_{\mathrm{F}},\|\Sigma_{2}\|_{\mathrm{F}})\leq\|B_{\mu}\|_{\mathrm{F}}\leq\frac{1}{2}\|\Sigma_{1}\|_{\mathrm{F}}+\frac{1}{2}\|\Sigma_{2}\|_{\mathrm{F}}.\]
As we prove later in Section \ref{smuexamplesproofs},
\[\max\left(0.4\|\Sigma_{1}\|_{\mathrm{op}},\,0.4\|\Sigma_{2}\|_{\mathrm{op}},\,\frac{1}{2}\|\Sigma_{1}-\Sigma_{2}\|_{\mathrm{F}}\right)\leq s(\mu)\leq\|\Sigma_{1}\|_{\mathrm{op}}+\|\Sigma_{2}\|_{\mathrm{op}}+\frac{1}{2}\|\Sigma_{1}-\Sigma_{2}\|_{\mathrm{F}},\]
where $\|\,\|_{\mathrm{op}}$ denotes the spectral norm. Thus if $\|\Sigma_{1}-\Sigma_{2}\|_{\mathrm{F}}>\epsilon(\|\Sigma_{1}\|_{\mathrm{F}}+\|\Sigma_{2}\|_{\mathrm{F}})$ (i.e., $\Sigma_{1}$ and $\Sigma_{2}$ differ enough), then $\mu$ satisfies the mixture model hypothesis in Problem \ref{mainproblem}, where $0<\epsilon<1$ is the threshold.
We remark that although the two components $\mathcal{N}(0,\Sigma_{1})$ and $\mathcal{N}(0,\Sigma_{2})$ have the same center $0$, the quantity $s(\mu)$ can still detect that $\mu$ satisfies mixture model hypothesis as long as $\Sigma_{1}$ and $\Sigma_{2}$ differ enough. This is unlike certain clustering objectives (e.g., $k$-means) which cannot separate two clusters with the same center.
\end{example}

\subsection{Computability issue}
The main issue with testing the mixture model hypothesis in Problem \ref{mainproblem} is that it is not clear how to compute $s(\mu)$ in polynomial time.  
This is true even if $\mu$ is supported on, say, $n$ points.  Note that while the objective in (\ref{defsmueq}) is convex in $\mu_{1}$ and $\mu_{2}$, we are taking a supremum rather than an infimum in (\ref{defsmueq}).

In fact, it is no coincidence that most classical quantities that quantify the structure of a given measure $\mu$ (e.g., the Cheeger constant \cite{KL} and the subgaussian constant \cite[Definition 3.4.1]{Romanbook}) share the same issue where it is not clear how to compute the quantities in polynomial time. 
Intuitively, the issue stems from the fact that these quantities are all orthogonally invariant in $\mu$ (i.e., invariant under pushforward of $\mu$ by any orthogonal transformation on $\mathbb{R}^{d}$). With an orthogonally invariant quantity in $\mu$, one needs to take into account the given measure $\mu$ viewed from every direction; and often exponentially many directions are needed in order to obtain an accurate approximation of the quantity. As a result, direct computation of such quantities often requires an exponential amount of time as the dimension $d$ gets large. 
For example, the subgaussian constant of a given probability measure $\mu$ on $\mathbb{R}^{d}$ is given by $\sup_{y\in\mathbb{S}^{d-1}}\|\langle X,y\rangle\|_{\psi_{2}}$, where $\mathbb{S}^{d-1}$ is the unit sphere in $\mathbb{R}^{d}$, $X$ is a random vector in $\mathbb{R}^{d}$ distributed according to $\mu$, and $\|\,\|_{\psi_{2}}$ denotes the subgaussian norm. If one attempts to estimate this supremum directly, one needs exponentially many points in $\mathbb{S}^{d-1}$ in order to accurately cover  $\mathbb{S}^{d-1}$.

Therefore, in most cases, an orthogonally invariant quantity in $\mu$ has the issue where it not clear how to compute it in polynomial time. However, in the next subsection, we introduce a new set of quantities in $\mu$ that are both computable and orthogonally invariant. Later, we show how these invariants can be used to give nonasymptotic bounds for Problem \ref{mainproblem}, and furnish a complete resolution of Problem \ref{mainproblemasymptotic} under the $L^{8}$-$L^{2}$ equivalence assumption.

\subsection{The 4th moment operator and its eigenvalues}
Let us start by observing that the first, second and third order statistics of a given probability measure $\mu$ on $\mathbb{R}^{d}$ do not give enough information to determine whether $\mu$ satisfies the mixture model hypothesis. Indeed, consider the single normal distribution $\mu=\mathcal{N}(0,\frac{1}{2}I)$ on $\mathbb{R}^d$ and the mixture distribution $\nu=\frac{1}{2}\mathcal{N}(0,P)+\frac{1}{2}\mathcal{N}(0,I-P)$, where $P\in\mathbb{R}^{d\times d}$ is any orthogonal projection. The measures $\mu$ and $\nu$ have exactly the same first, second and third order statistics: the mean vector is $0$, the covariance matrix is $\frac{1}{2}I$, and the third moment tensor $\int_{\mathbb{R}^{d}}x\otimes x\otimes x\,d\mu(x)=\int_{\mathbb{R}^{d}}x\otimes x\otimes x\,d\nu(x)=0$ in the tensor space $\mathbb{R}^{d}\otimes\mathbb{R}^{d}\otimes\mathbb{R}^{d}$, since $\mu$ and $\nu$ are both symmetric about the origin. On the other hand, by Example \ref{introsinglenormal} and Example \ref{intromixture} when $d$ is large enough, $\mu$ satisfies the single distribution hypothesis, whereas $\nu$ satisfies the mixture model hypothesis. Hence, these two measures $\mu$ and $\nu$ illustrate that the first, second and third order statistics are insufficient to decide whether a given measure satisfies the mixture model hypothesis.

It turns out that the 2nd and 4th order statistics together do suffice to test the mixture model hypothesis, at least asymptotically, as we prove later in this paper. For a given probability measure $\mu\in\mathcal{P}_{4}(\mathbb{R}^{d})$, its 4th order statistics are encoded in its 4th moment tensor $\int_{\mathbb{R}^{d}}x\otimes x\otimes x\otimes x\,d\mu(x)$ in the tensor space $\mathbb{R}^{d}\otimes\mathbb{R}^{d}\otimes\mathbb{R}^{d}\otimes\mathbb{R}^{d}$. It should be stressed that while the structure of a matrix can be understood via its spectral decomposition, most tensor problems are NP hard as the dimension $d$ gets large \cite{NPhard}.  Hence one might suspect that it may not be a good approach to test the mixture model hypothesis by directly analyzing the tensor structure of the 4th moment tensor.

 We introduce a new set of quantities in $\mu$ that are both computable and orthogonally invariant,  namely the eigenvalues of the operator $T_{\mu}$, defined below using 4th order statistics. (See also Remark \ref{computableinvariantproof} below.) The  operator  $T_{\mu}$ is, in fact, the  flattening of the 4th moment tensor of $\mu$ into an operator. This flattening of the tensor into an operator avoids the issue of NP hardness of tensor problems.
\begin{definition} \label{def:T_mu}
Let $\mathbb{R}_{\mathrm{sym}}^{d\times d}$ be the inner product space of all $d\times d$ real symmetric matrices with the inner product defined by
\[\langle A,B\rangle:=\mathrm{Tr}(AB),\quad A,B\in\mathbb{R}_{\mathrm{sym}}^{d\times d},\]
where $\mathrm{Tr}(\,)$ denotes the trace of the underlying matrix. For a probability measure $\mu\in\mathcal{P}_{4}(\mathbb{R}^{d})$, its {\it 4th moment operator} $T_{\mu}:\mathbb{R}_{\mathrm{sym}}^{d\times d}\to\mathbb{R}_{\mathrm{sym}}^{d\times d}$ is defined by
\[T_{\mu}(A):=\int_{\mathbb{R}^d}\langle A,xx^{T}\rangle xx^{T}\,d\mu(x),\]
for all $A\in\mathbb{R}_{\mathrm{sym}}^{d\times d}$, where $x^{T}$ is the transpose of $x$.
\end{definition}
\begin{remark}
One can write $T_{\mu}$ as the covariance operator $T_{\mu}(A)=\int_{\mathbb{R}_{\mathrm{sym}}^{d\times d}}\langle A,F\rangle F\,d\nu(F)$ associated with the measure $\nu$ on $\mathbb{R}_{\mathrm{sym}}^{d\times d}$, where $\nu$ is the pushforward measure of $\mu$ by the map $x\mapsto xx^{T}$. However, this covariance operator is not a typical covariance operator; unlike most distributions studied in the high dimensional probability literature, the pushforward measure $\nu$ is supported on the manifold $\{xx^{T}|\,x\in\mathbb{R}^{d}\}$. This is quite a special manifold. Indeed, the manifold structure makes it so the largest eigenvalue of $T_{\mu}$ is always an outlier eigenvalue as we see in Theorem \ref{firstmain} below.
\end{remark}
\begin{remark}\label{tmupsd}
Since $\langle T_{\mu}(A),B\rangle=\langle A,T_{\mu}(B)\rangle$ and $\langle T_{\mu}(A),A\rangle\geq 0$ for all $A,B\in\mathbb{R}_{\mathrm{sym}}^{d\times d}$, the operator $T_{\mu}$ is positive semidefinite.
\end{remark}
\begin{remark}\label{computableinvariantproof}
The eigenvalues of the operator $T_{\mu}$ are computable in polynomial time as long as one has knowledge of all the 4th order statistics of $\mu$. In practice, one may have to take $n$ samples of $\mu$ to form an empirical measure $\mu^{(n)}$ and use $T_{\mu^{(n)}}$ to approximate $T_{\mu}$.

The eigenvalues of $T_{\mu}$ are also invariant under pushforward of $\mu$ by any orthogonal transformation $M$ on $\mathbb{R}^{d}$. Indeed, let $M_{\#}\mu$ be the pushforward of $\mu$ by $M$. Then for every $A\in\mathbb{R}_{\mathrm{sym}}^{d\times d}$, we have
\begin{eqnarray*}
T_{M_{\#}\mu}(A)&=&\int_{\mathbb{R}^{d}}\langle A,xx^{T}\rangle xx^{T}\,d(M_{\#}\mu)(x)\\&=&
\int_{\mathbb{R}^{d}}\langle A,(Mx)(Mx)^{T}\rangle\,(Mx)(Mx)^{T}\,d\mu(x)\\&=&
\int_{\mathbb{R}^{d}}\langle M^{T}AM,xx^{T}\rangle Mxx^{T}M^{T}\,d\mu(x)\\&=&
L\circ T_{\mu}\circ L^{-1}(A),
\end{eqnarray*}
where $L:\mathbb{R}_{\mathrm{sym}}^{d\times d}\to\mathbb{R}_{\mathrm{sym}}^{d\times d}$ is defined by $L(A)=MAM^{T}$ for $A\in\mathbb{R}_{\mathrm{sym}}^{d\times d}$, and $\circ$ denotes the composition of the operators. Hence, the operators $T_{M_{\#}\mu}$ and $T_{\mu}$ are similar and thus have the same eigenvalues.
\end{remark}

We now go through some examples of $\mu$ for which we can explicitly list out the eigenvalues of $T_{\mu}$. The actual computations and proofs are done in Section \ref{examplesection}.  In these examples, we always have in mind that the dimension $d$ is large. In all the examples, notice that the largest eigenvalue of $T_{\mu}$ is significantly larger than most of the other eigenvalues of $T_{\mu}$. Hence, a \textit{spectral gap} appears in the eigenvalue distribution of $T_{\mu}$.

Since $\mathrm{dim}\,\mathbb{R}_{\mathrm{sym}}^{d\times d}=\frac{d(d+1)}{2}$, the operator $T_{\mu}$ always has $\frac{d(d+1)}{2}$ eigenvalues, which we denote by $\lambda_{1}(T_{\mu})\geq\lambda_{2}(T_{\mu})\geq\ldots\geq\lambda_{\frac{d(d+1)}{2}}(T_{\mu})$.
\begin{example}\label{examplelambdastandardnormal}
Suppose that $\mu=\mathcal{N}(0,I)$ is the standard normal distribution on $\mathbb{R}^{d}$. Then
\[\lambda_{i}(T_{\mu})=\begin{cases}d+2,&i=1\\2,&i\geq 2\end{cases},\]
for all $1\leq i\leq\frac{d(d+1)}{2}$. In this example, the largest eigenvalue is of order $d$, whereas the rest of the eigenvalues are of order $1$.
\end{example}
\begin{example}\label{examplelambdaiid}
Suppose that $X_{1},\ldots,X_{d}$ are independent identically distributed random variables with $\mathbb{E}X_{1}=0$, $\mathbb{E}X_{1}^{2}=1$ and $\mathbb{E}X_{1}^{4}<\infty$. Let $\mu$ be the distribution of the random vector $X=(X_{1},\ldots,X_{d})^{T}$ in $\mathbb{R}^{d}$. Then the $\frac{d(d+1)}{2}$ eigenvalues of $T_{\mu}$ are listed as follows:
\[\underbrace{d+\mathbb{E}X_{1}^{4}-1}_{1},\;\underbrace{\mathbb{E}X_{1}^{4}-1,\ldots,\mathbb{E}X_{1}^{4}-1}_{d-1},\;\underbrace{2,\ldots,2}_{\frac{d(d-1)}{2}}.\]This example is a generalization of Example \ref{examplelambdastandardnormal}. Like in the above example, the largest eigenvalue is of order $d$, whereas the rest of the eigenvalues are of order $1$. So a spectral gap occurs between the first and second largest eigenvalues of $T_{\mu}$.
\end{example}
\begin{example}\label{examplelambdamixture}
Suppose that $\mathbb{R}^{d}=V_{1}\oplus\ldots\oplus V_{r}$ is an orthogonal decomposition, i.e., $V_{1},\ldots,V_{r}$ are subspaces of $\mathbb{R}^{d}$ such that $\mathbb{R}^{d}=V_{1}+\ldots+V_{r}$ and $V_{i}\perp V_{j}$ for all $i\neq j$ in $\{1,\ldots,r\}$. Assume further that all the $V_{i}$ have the same dimension $\frac{d}{r}\in\mathbb{N}$. For each $i\in\{1,\ldots,r\}$, let $P_{i}$ be the orthogonal projection from $\mathbb{R}^{d}$ onto $V_{i}$. Let
\[\mu=\frac{1}{r}\mathcal{N}(0,P_{1})+\ldots+\frac{1}{r}\mathcal{N}(0,P_{r})\]
be the equal weight mixture of the distributions $\mathcal{N}(0,P_{1}),\ldots,\mathcal{N}(0,P_{r})$. Then
\[\lambda_{i}(T_{\mu})=\begin{cases}\frac{d+2r}{r^{2}},&1\leq i\leq r\\\frac{2}{r},&r+1\leq r\leq\frac{d(d+r)}{2r}\\0,&i>\frac{d(d+r)}{2r}\end{cases},\]
for all $1\leq i\leq\frac{d(d+1)}{2}$. In this example, the largest $r$ eigenvalues are of order $d$, whereas the rest of the eigenvalues are of order $1$ or $0$. Thus, a spectral gap occurs between the $r$th largest and the $(r+1)$th largest eigenvalues.
\end{example}
Observe that in Example \ref{examplelambdastandardnormal} where $\mu$ is a single normal distribution, there is a spectral gap between the first and second largest eigenvalues $\lambda_{1}(T_{\mu})$ and $\lambda_{2}(T_{\mu})$. Meanwhile, in Example \ref{examplelambdamixture} where $\mu$ is a mixture distribution, we have $\lambda_{1}(T_{\mu})=\lambda_{2}(T_{\mu})$, assuming $r\geq 2$. This phenomenon resembles the classical fact in spectral graph theory that a regular graph is disconnected if and only if the first and second largest eigenvalues of
its adjacency matrix coincide. As we see in Theorem \ref{thm:spectral_decomposition} and Corollary \ref{asymptoticresolution} below, such a phenomenon holds in general, where we show how the two eigenvalues $\lambda_{1}(T_{\mu})$ and $\lambda_{2}(T_{\mu})$ can be used to test the mixture model hypothesis for $\mu$.  
Our paper thus makes a new link between high dimensional probability and spectral graph theory.

\subsection{Main results}
We now state the first main result of this paper. It says that the largest eigenvalue $\lambda_{1}(T_{\mu})$ is at least $\frac{d+1}{2}$ times the average of the $\frac{d(d+1)}{2}$ eigenvalues of $T_{\mu}$. Moreover, $\lambda_{1}(T_{\mu})$ is as large as the $\ell^{2}$-sum of all the eigenvalues of $T_{\mu}$, if $\mu$ satisfies the $L^{4}$-$L^{2}$ equivalence condition, which is a commonly used assumption in high dimensional probability (see Appendix \ref{lpl2section}).
\pagebreak

\begin{theorem}\label{firstmain}
Let $\mu\in\mathcal{P}_{4}(\mathbb{R}^{d})$. Then
\begin{equation}\label{firstmaineq1}
\lambda_{1}(T_{\mu})\geq\frac{1}{d}\sum_{i=1}^{d(d+1)/2}\lambda_{i}(T_{\mu}).
\end{equation}
Moreover, if $\beta\geq 1$ satisfies
\begin{equation}\label{l4l2}
\left(\int_{\mathbb{R}^d}|\langle x,v\rangle|^{4}\,d\mu(x)\right)^{1/4}\leq\beta\left(\int_{\mathbb{R}^d}|\langle x,v\rangle|^{2}\,d\mu(x)\right)^{{1/2}},
\end{equation}
for all $v\in\mathbb{R}^{d}$, then
\begin{equation}\label{firstmaineq2}
\|B_{\mu}\|_{\mathrm{F}}^{2}\leq\lambda_{1}(T_{\mu})\leq\left(\sum_{i=1}^{d(d+1)/2}\lambda_{i}(T_{\mu})^{2}\right)^{1/2}\leq\beta^{4}\|B_{\mu}\|_{\mathrm{F}}^{2}.
\end{equation}
\end{theorem}

\begin{remark}
The coefficient $\frac{1}{d}$ on the right hand side of (\ref{firstmaineq1}) is optimal, since when $\mu=\mathcal{N}(0,I)$, by Example \ref{examplelambdastandardnormal}, we have $\lambda_{1}(T_{\mu})=d+2$ and $\sum_{i=1}^{d(d+1)/2}\lambda_{i}(T_{\mu})=d(d+2)$.
\end{remark}

The second main result of this paper gives an upper bound and a lower bound for $s(\mu)$ in terms of $\lambda_{1}(T_{\mu}),\lambda_{2}(T_{\mu})$ and $\|B_{\mu}\|_{\mathrm{F}}$ under the $L^{8}$-$L^{2}$ equivalence assumption (see Appendix \ref{lpl2section}).

\begin{theorem} \label{thm:spectral_decomposition}
Suppose that $\mu\in\mathcal{P}_{8}(\mathbb{R}^{d})$, $\mu\neq\delta_{0}$ and $\beta\geq 1$ satisfy
\begin{equation}\label{l8l2}
\left(\int_{\mathbb{R}^d}|\langle x,v\rangle|^{8}\,d\mu(x)\right)^{1/8}\leq\beta\left(\int_{\mathbb{R}^d}|\langle x,v\rangle|^{2}\,d\mu(x)\right)^{{1/2}},
\end{equation}
for all $v\in\mathbb{R}^{d}$. Then
\begin{equation}\label{secondmaineq1}
\frac{1}{200\beta^{8}}\cdot\left[\frac{\lambda_{2}(T_{\mu})}{\lambda_{1}(T_{\mu})}+\left(1-\frac{\|B_{\mu}\|_{\mathrm{F}}^{2}}{\lambda_{1}(T_{\mu})}\right)\right]^{3}\leq\left(\frac{s(\mu)}{\|B_{\mu}\|_{\mathrm{F}}}\right)^{2}\leq 4\left[\frac{\lambda_{2}(T_{\mu})}{\lambda_{1}(T_{\mu})}+\left(1-\frac{\|B_{\mu}\|_{\mathrm{F}}^{2}}{\lambda_{1}(T_{\mu})}\right)\right].
\end{equation}
\end{theorem}

\begin{remark}
Since the upper bound in (\ref{secondmaineq1}) does not involve $\beta$, the upper bound does not actually require the assumption (\ref{l8l2}).  Only the lower bound needs (\ref{l8l2}).
\end{remark}

\begin{remark}
Each of the three terms $\frac{\lambda_{2}(T_{\mu})}{\lambda_{1}(T_{\mu})}$, $1-\frac{\|B_{\mu}\|_{\mathrm{F}}^{2}}{\lambda_{1}(T_{\mu})}$ and $\frac{s(\mu)}{\|B_{\mu}\|_{\mathrm{F}}}$ that appear in (\ref{secondmaineq1}) is between $0$ and $1$. To see this, note that $0\leq\frac{\lambda_{2}(T_{\mu})}{\lambda_{1}(T_{\mu})}\leq 1$ is obvious; that $0\leq 1-\frac{\|B_{\mu}\|_{\mathrm{F}}^{2}}{\lambda_{1}(T_{\mu})}\leq 1$ follows from Theorem \ref{firstmain}; and that $0\leq\frac{s(\mu)}{\|B_{\mu}\|_{\mathrm{F}}}\leq 1$ follows from (\ref{eq:s_remark}).
\end{remark}

\begin{remark}
Our proof of the lower bound in (\ref{secondmaineq1}) is constructive. 
Specifically, we use a quadratic separation to construct a mixture decomposition of $\mu$. 
This is unlike the commonly used linear hyperplane separation in the literature (e.g., $k$-means with $k=2$ and KLS conjecture). For example, when $\mu=\frac{1}{2}\mathcal{N}(0,I)+\frac{1}{2}\mathcal{N}(0,2I)$, the first component $\mathcal{N}(0,I)$ is concentrated around the sphere $\{x\in\mathbb{R}^{d}:\,\|x\|_{2}=\sqrt{d}\}$, whereas the second component $\mathcal{N}(0,2I)$ is concentrated around the sphere $\{x\in\mathbb{R}^{d}:\,\|x\|_{2}=\sqrt{2d}\}$. The most natural way to separate these two spheres is to separate along a middle sphere: $\|x\|_{2}<1.1\sqrt{d}$ and $\|x\|_{2}>1.1\sqrt{d}$ rather than separating along a hyperplane. For more details of the construction of the mixture decomposition of $\mu$ in proving the lower bound in (\ref{secondmaineq1}), we refer to Proposition \ref{secondmainlbexplicit} below.
\end{remark}

Next, we give a complete resolution of Problem \ref{mainproblemasymptotic} in the case when the measures satisfy the $L^{8}$-$L^{2}$ assumption. This result follows immediately from Theorem \ref{thm:spectral_decomposition}.
\begin{corollary}\label{asymptoticresolution}
Let $(\mu^{(n)})_{n\in\mathbb{N}}$ be a sequence of probability measures with $\mu^{(n)}\in\mathcal{P}_{8}(\mathbb{R}^{d_{n}})$, where $d_{n}\in\mathbb{N}$, and $\mu^{(n)}\neq\delta_{0}$. Suppose that there exists $1\leq\beta<\infty$ satisfying
\[\left(\int_{\mathbb{R}^d}|\langle x,v\rangle|^{8}\,d\mu^{(n)}(x)\right)^{1/8}\leq\beta\left(\int_{\mathbb{R}^d}|\langle x,v\rangle|^{2}\,d\mu^{(n)}(x)\right)^{{1/2}},\]
for all $v\in\mathbb{R}^{d}$ and $n\in\mathbb{N}$. Then the following statements are equivalent:
\begin{enumerate}[(1)]
\item $\displaystyle\lim_{n\to\infty}\frac{s(\mu^{(n)})}{\|B_{\mu^{(n)}}\|_{\mathrm{F}}}=0$;
\item $\displaystyle\lim_{n\to\infty}\frac{\lambda_{2}(T_{\mu^{(n)}})}{\lambda_{1}(T_{\mu^{(n)}})}=0$ and $\displaystyle\lim_{n\to\infty}\frac{\|B_{\mu^{(n)}}\|_{\mathrm{F}}^{2}}{\lambda_{1}(T_{\mu^{(n)}})}=1$.
\end{enumerate}
\end{corollary}

\begin{remark}
In the statement (2) in Corollary \ref{asymptoticresolution}, the first condition $\lim_{n\to\infty}\frac{\lambda_{2}(T_{\mu^{(n)}})}{\lambda_{1}(T_{\mu^{(n)}})}=0$ means that the second largest eigenvalue is significantly smaller than the first largest eigenvalue, but the second condition $\lim_{n\to\infty}\frac{\|B_{\mu^{(n)}}\|_{\mathrm{F}}^{2}}{\lambda_{1}(T_{\mu^{(n)}})}=1$ may seem less intuitive. To see why this condition is essential, fix a measure $\mu\in\mathcal{P}_{8}(\mathbb{R}^{d})$ and consider the measure $\nu=\frac{1}{2}\mu+\frac{1}{2}\delta_{0}$. Since $\nu=\frac{1}{2}\mu+\frac{1}{2}\delta_{0}$, we have $s(\nu)\geq\frac{1}{2}\|B_{\mu}-B_{\delta_{0}}\|_{\mathrm{F}}=\frac{1}{2}\|B_{\mu}\|_{\mathrm{F}}$. Note that $B_{\nu}=\frac{1}{2}B_{\mu}$. Therefore, $s(\nu)\geq\|B_{\nu}\|_{\mathrm{F}}$, so by (\ref{eq:s_remark}), we have $s(\nu)=\|B_{\nu}\|_{\mathrm{F}}$, or equivalently, $\frac{s(\nu)}{\|B_{\nu}\|_{\mathrm{F}}}=1$.
On the other hand, since $\nu=\frac{1}{2}\mu+\frac{1}{2}\delta_{0}$, we have $T_{\nu}=\frac{1}{2}T_{\mu}$ and $B_{\nu}=\frac{1}{2}B_{\mu}$. So
\[\frac{\lambda_{2}(T_{\nu})}{\lambda_{1}(T_{\nu})}=\frac{\lambda_{2}(T_{\mu})}{\lambda_{1}(T_{\mu})}\quad\text{and}\quad\frac{\|B_{\nu}\|_{\mathrm{F}}^{2}}{\lambda_{1}(T_{\nu})}=\frac{1}{2}\cdot\frac{\|B_{\mu}\|_{\mathrm{F}}^{2}}{\lambda_{1}(T_{\mu})}.\]
From this, we can see that merely the ratio $\frac{\lambda_{2}(T_{\mu})}{\lambda_{1}(T_{\mu})}$ does not give enough information about $\frac{s(\mu)}{\|B_{\mu}\|_{\mathrm{F}}}$. Thus, it is essential to also consider the quantity $\frac{\|B_{\mu}\|_{\mathrm{F}}^{2}}{\lambda_{1}(T_{\mu})}$.
\end{remark}

\begin{remark}\label{spectralgraphresemble}
Roughly speaking, Corollary \ref{asymptoticresolution} says that a probability measure $\mu$ satisfies the single distribution hypothesis if and only if $\lambda_{1}(T_{\mu})\gg\lambda_{2}(T_{\mu})$ and $\frac{\|B_{\mu}\|_{\mathrm{F}}^{2}}{\lambda_{1}(T_{\mu})}\approx 1$. This is analogous to the famous fact in spectral graph theory that a regular graph is connected if and only if the first largest eigenvalue of its adjacency matrix is strictly larger than the second largest eigenvalue.
\end{remark}

\subsection{Some variants of the second order separation parameter}\label{unequalsubsection}
Here we mention two variants of the quantity $s(\mu)$. For the first variant, we change equal weight mixtures to unequal weight mixtures. For the second one, we take into account the difference between the first order statistics of the two subpopulations.

Recall that in the definition of $s(\mu)$, the supremum in (\ref{defsmueq}) is taken over all possible equal weight mixture decompositions $\mu=\frac{1}{2}\mu_{1}+\frac{1}{2}\mu_{2}$. The following lemma says that if we modify the definition of $s(\mu)$ so that the supremum is taken over unequal weight mixture decompositions, the resulting quantity differs by at most a constant factor that depends on the weights.
\begin{lemma}\label{unequal}
Let $\mu\in\mathcal{P}_{2}(\mathbb{R}^{d})$ and $0<\alpha\leq\frac{1}{2}$. Then
\[\frac{1}{2(1-\alpha)}s(\mu)\leq\sup_{\mu=\alpha\nu_{1}+(1-\alpha)\nu_{2}}\|B_{\nu_{1}}-B_{\nu_{2}}\|_{\mathrm{F}}\leq\frac{1}{2\alpha}s(\mu),\]
where the supremum is over all probability measures $\nu_{1},\nu_{2}$ on $\mathbb{R}^{d}$ such that $\mu=\alpha\nu_{1}+(1-\alpha)\nu_{2}$.
\end{lemma}

Next, in the definition of $s(\mu)$, the term inside the supremum in (\ref{defsmueq}) only takes into account the difference between the second order statistics of $\mu_{1}$ and $\mu_{2}$. If we wish to take into account the difference in the first order statistics of $\mu_{1}$ and $\mu_{2}$, we can replace the measure $\mu$ by the product measure $\delta_{a}\times\mu$ on $\mathbb{R}^{d+1}$, where $a\in\mathbb{R}$ is fixed. Indeed,
\begin{align*}
s(\delta_{a}\times\mu)=&
\frac{1}{2}\cdot\sup_{\mu=\frac{1}{2}\mu_{1}+\frac{1}{2}\mu_{2}}\|B_{\delta_{a}\times\mu_{1}}-B_{\delta_{a}\times\mu_{2}}\|_{\mathrm{F}}\\=&
\frac{1}{2}\cdot\sup_{\mu=\frac{1}{2}\mu_{1}+\frac{1}{2}\mu_{2}}\left\|\int_{\mathbb{R}^d}\begin{bmatrix}a\\x\end{bmatrix}\begin{bmatrix}a\\x\end{bmatrix}^{T}\,d\mu_{1}(x)-\int_{\mathbb{R}^d}\begin{bmatrix}a\\x\end{bmatrix}\begin{bmatrix}a\\x\end{bmatrix}^{T}\,d\mu_{2}(x)\right\|_{\mathrm{F}}\\=&
\frac{1}{2}\cdot\sup_{\mu=\frac{1}{2}\mu_{1}+\frac{1}{2}\mu_{2}}\bigg(\left\|\int_{\mathbb{R}^{d}}xx^{T}\,d\mu_{1}(x)-\int_{\mathbb{R}^{d}}xx^{T}\,d\mu_{2}(x)\right\|_{\mathrm{F}}^{2}\\&
+2a^{2}\left\|\int_{\mathbb{R}^{d}}x\,d\mu_{1}(x)-\int_{\mathbb{R}^{d}}x\,d\mu_{2}(x)\right\|_{2}^{2}\bigg)^{1/2}.
\end{align*}
Here the supremum is over all probability measures $\mu_{1},\mu_{2}$ on $\mathbb{R}^{d}$ such that $\mu=\frac{1}{2}\mu_{1}+\frac{1}{2}\mu_{2}$.

\subsection{Notation}

Recall that $\mathcal{P}_{p}(\mathbb{R}^{d})$, $\delta_{x}$, $\|\,\|_{2}$ are defined at the beginning of the paper.

Throughout this work, all inner product spaces are over the field $\mathbb{R}$. We use the same notation $\langle\cdot,\cdot\rangle$ for both the canonical inner product on $\mathbb{R}^{d}$ and the inner product on $\mathbb{R}_{\mathrm{sym}}^{d\times d}$ defined in Definition \ref{def:T_mu}. In particular, $\langle A,xx^{T}\rangle=\langle Ax,x\rangle$ for all $A\in\mathbb{R}_{\mathrm{sym}}^{d\times d}$ and $x\in\mathbb{R}^{d}$.

If $\mathcal{H}$ is a finite dimensional inner product space and $z \in\mathcal{H}$, then $z \otimes z$ is the rank one linear operator on $\mathcal{H}$ defined by $v \mapsto \langle v,z\rangle z$ for $v \in\mathcal{H}$.

If $T:\mathcal{H}\to\mathcal{H}$ is a linear operator, then $\|T\|_{\mathrm{op}}:=\sup_{x\in\mathcal{H}\backslash\{0\}}\frac{\|Tx\|}{\|x\|}$ is the operator norm. If moreover, $T$ is self-adjoint, then $\lambda_{1}(T)\geq\ldots\geq\lambda_{\mathrm{dim}\,\mathcal{H}}(T)$ are the eigenvalues of $T$ in descending order. Note that if $T$ is positive semidefinite, then $\lambda_{1}(T)=\|T\|_{\mathrm{op}}$.

For $1\leq i\leq d$, denote by $e_{i}$ the vector in $\mathbb{R}^{d}$ with the $i$th entry equal to $1$ and all other entries equal to $0$. The $d\times d$ identity matrix is denoted by $I$. If $M$ is a $d\times d$ matrix, then $\|M\|_{\mathrm{F}}:=\sum_{i=1}^{d}\sum_{j=1}^{d}|\langle Me_{i},e_{j}\rangle|^{2}$ is the Frobernius norm of $M$, and $\mathrm{Tr}(M):=\sum_{i=1}^{d}\langle Me_{i},e_{i}\rangle$ is the trace of $M$. The transpose of $M$ is denoted by $M^{T}$.

If $w\in\mathbb{R}^{d}$ and $\Sigma$ is a $d\times d$ positive semidefinite matrix, then $\mathcal{N}(w,\Sigma)$ denotes the normal distribution on $\mathbb{R}^{d}$ with mean $w$ and covariance $\Sigma$.

\subsection{Organization of the paper} The rest of this paper is organized as follows. In Section \ref{1stmainproofsection}, we prove the first main result Theorem \ref{firstmain}. In Section \ref{2ndmainproofsection}, we prove the second main result Theorem \ref{thm:spectral_decomposition}. In Section \ref{smuexamplesproofs}, we prove the bounds for $s(\mu)$ stated in Example \ref{introproduct}-Example \ref{introgeneralmixture}. In Section \ref{examplesection}, we prove Examples \ref{examplelambdastandardnormal}-\ref{examplelambdamixture}. In Section \ref{unequalsection}, we prove Lemma \ref{unequal}. In Section \ref{discuss}, we discuss some possible future directions.

\section{Proof of the first main result}\label{1stmainproofsection}
In this section, we prove the first main result Theorem \ref{firstmain}.
\begin{lemma} \label{lem:l2norm_equality}
Let $\mu\in\mathcal{P}_{4}(\mathbb{R}^{d})$. Then 
\[
\int_{\mathbb{R}^d}\langle x,v\rangle^{2}\,d\mu(x) = \|B_{\mu}^{1/2}v\|_2^2,
\]
for all $v\in\mathbb{R}^{d}$.
\end{lemma}
\begin{proof} We have 
\begin{align*}
\int_{\mathbb{R}^d}\langle x,v\rangle^{2}\,d\mu(x) 
&= \int_{\mathbb{R}^d}\langle xx^T,vv^T\rangle\,d\mu(x) \\
&= \langle B_{\mu},vv^T\rangle 
= \langle B_{\mu}v,v\rangle
= \langle B_{\mu}^{1/2}v,B_{\mu}^{1/2}v\rangle = \|B_{\mu}^{1/2}v\|_2^2,
\end{align*}
as claimed.
\end{proof}

\begin{lemma}\label{lpl2advance}
Suppose that $p\geq 2$, $\mu\in\mathcal{P}_{p}(\mathbb{R}^{d})$ and $\beta\geq 1$ satisfy
\begin{equation}\label{lpl2}
\left(\int_{\mathbb{R}^d}|\langle x,v\rangle|^{p}\,d\mu(x)\right)^{1/p}\leq\beta\left(\int_{\mathbb{R}^d}|\langle x,v\rangle|^{2}\,d\mu(x)\right)^{{1/2}},
\end{equation}
for all $v\in\mathbb{R}^{d}$. Then
\[\left(\int_{\mathbb{R}^{d}}\|Mx\|_{2}^{p}\,d\mu(x)\right)^{1/p}\leq\beta\|B_{\mu}^{\frac{1}{2}}M^{T}\|_{\mathrm{F}},\]
for all $d\times d$ matrix $M$.
\end{lemma}
\begin{proof}
Let $\{e_{1},\ldots,e_{d}\}$ be the canonical basis for $\mathbb{R}^{d}$. Then
\begin{eqnarray*}
\left(\int_{\mathbb{R}^{d}}\|Mx\|_{2}^{p}\,d\mu(x)\right)^{2/p}&=&
\left(\int_{\mathbb{R}^{d}}\left|\sum_{i=1}^{d}\langle Mx,e_{i}\rangle^{2}\right|^{p/2}\,d\mu(x)\right)^{2/p}\\&=&
\left(\int_{\mathbb{R}^{d}}\left|\sum_{i=1}^{d}\langle x,M^{T}e_{i}\rangle^{2}\right|^{p/2}\,d\mu(x)\right)^{2/p}\\&\leq&
\sum_{i=1}^{d}\left(\int_{\mathbb{R}^{d}}\left|\langle x,M^{T}e_{i}\rangle^{2}\right|^{p/2}\,d\mu(x)\right)^{2/p}\\&=&
\sum_{i=1}^{d}\left(\int_{\mathbb{R}^{d}}|\langle x,M^{T}e_{i}\rangle|^{p}\,d\mu(x)\right)^{2/p}\\&\leq&
\sum_{i=1}^{d}\beta^{2}\int_{\mathbb{R}^{d}}|\langle x,M^{T}e_{i}\rangle|^{2}\,d\mu(x)\\&=&
\sum_{i=1}^{d}\beta^{2}\|B_{\mu}^{1/2}M^{T}e_{i}\|_{2}^{2}\\&=&
\beta^{2}\|B_{\mu}^{1/2}M^{T}\|_{\mathrm{F}}^{2},
\end{eqnarray*}
where the third step follows from Minkowski's inequality and the assumption that $p\geq 2$, the fifth step follows from (\ref{lpl2}) and the sixth step follows from Lemma \ref{lem:l2norm_equality}.
\end{proof}
Before we prove Theorem \ref{firstmain}, let us make an observation. By definition, if $M \in \mathbb{R}_{\mathrm{sym}}^{d\times d}$, then $M \otimes M$ is the rank one linear operator on $\mathbb{R}_{\mathrm{sym}}^{d\times d}$ defined by $A \mapsto \inner{A}{M}M$ for $A \in \mathbb{R}_{\mathrm{sym}}^{d \times d}$. Hence, for every probability measure $\mu\in\mathcal{P}_{4}(\mathbb{R}^{d})$, we can write $T_{\mu}$ as an integral of rank one linear operators on $\mathbb{R}_{\mathrm{sym}}^{d \times d}$:
\begin{equation}\label{tmutensorform}
T_{\mu}=\int_{\mathbb{R}^{d}}(xx^{T})\otimes(xx^{T})\,d\mu(x).
\end{equation}
\begin{proof}[Proof of Theorem \ref{firstmain}]
We first prove (\ref{firstmaineq1}). We have
\[\lambda_{1}(T_{\mu})=\|T_{\mu}\|_{\mathrm{op}}\geq\frac{1}{\|I\|_{\mathrm{F}}^{2}}\langle T_{\mu}(I),I\rangle=\frac{1}{d}\int_{\mathbb{R}^{d}}\langle I,xx^{T}\rangle^{2}\,d\mu(x)=\frac{1}{d}\int_{\mathbb{R}^{d}}\|x\|_{2}^{4}\,d\mu(x).\]
On the other hand, since $T_{\mu}$ can be written as $\int_{\mathbb{R}^{d}}(xx^{T})\otimes(xx^{T})\,d\mu(x)$ by (\ref{tmutensorform}), the trace of the operator $T_{\mu}$ is
\[\sum_{i=1}^{d(d+1)/2}\lambda_{i}(T_{\mu})=\int_{\mathbb{R}^{d}}\langle xx^{T},xx^{T}\rangle\,d\mu(x)=\int_{\mathbb{R}^{d}}\|x\|_{2}^{4}\,d\mu(x),\]
where we use the fact that the trace operation is linear and so it can be interchanged with the integral. Therefore,
\[\lambda_{1}(T_{\mu})\geq\frac{1}{d}\int_{\mathbb{R}^{d}}\|x\|_{2}^{4}\,d\mu(x)=\frac{1}{d}\sum_{i=1}^{d(d+1)/2}\lambda_{i}(T_{\mu}).\]
This proves (\ref{firstmaineq1}).

To prove (\ref{firstmaineq2}), observe that for every $A\in\mathbb{R}_{\mathrm{sym}}^{d\times d}$, we have
\[\langle T_{\mu}(A),A\rangle=\int_{\mathbb{R}^{d}}\langle A,xx^{T}\rangle^{2}\,d\mu(x)\geq\left(\int_{\mathbb{R}^{d}}\langle A,xx^{T}\rangle\,d\mu(x)\right)^{2}=\langle A,B_{\mu}\rangle^{2}.\]
Hence, taking supremum over all $A\in\mathbb{R}_{\mathrm{sym}}^{d\times d}$ with $\|A\|_{\mathrm{F}}\leq 1$, we have $\|T_{\mu}\|_{\mathrm{op}}\geq\|B_{\mu}\|_{\mathrm{F}}^{2}$. This proves the first inequality in (\ref{firstmaineq2}).

To prove the last inequality in (\ref{firstmaineq2}), observe that since $T_{\mu}$ is positive semidefinite (see Remark \ref{tmupsd}), we have $\lambda_{i}(T_{\mu})^{2}=\lambda_{i}(T_{\mu}^{2})$ for each $1\leq i\leq\frac{d(d+1)}{2}$, and so
\begin{equation}\label{firstmainproofeq1}
\sum_{i=1}^{d(d+1)/2}\lambda_{i}(T_{\mu})^{2}=\sum_{i=1}^{d(d+1)/2}\lambda_{i}(T_{\mu}^{2})
\end{equation}
is equal to the trace of the operator $T_{\mu}^{2}$. Since $T_{\mu}=\int_{\mathbb{R}^{d}}(xx^{T})\otimes(xx^{T})\,d\mu(x)$ by (\ref{tmutensorform}), we have
\[T_{\mu}^{2}=\int_{\mathbb{R}^{d}}\int_{\mathbb{R}^{d}}\langle xx^{T},yy^{T}\rangle(xx^{T})\otimes(yy^{T})\,d\mu(x)\,d\mu(y),\]
so the trace of $T_{\mu}^{2}$ is equal to
\[\int_{\mathbb{R}^{d}}\int_{\mathbb{R}^{d}}\langle xx^{T},yy^{T}\rangle^{2}\,d\mu(x)\,d\mu(y)=\int_{\mathbb{R}^{d}}\int_{\mathbb{R}^{d}}\langle x,y\rangle^{4}\,d\mu(x)\,d\mu(y).\]
Thus, by (\ref{firstmainproofeq1}),
\begin{equation}\label{firstmainproofeq2}
\sum_{i=1}^{d(d+1)/2}\lambda_{i}(T_{\mu})^{2}=\int_{\mathbb{R}^{d}}\int_{\mathbb{R}^{d}}\langle x,y\rangle^{4}\,d\mu(x)\,d\mu(y).
\end{equation}
We now bound this quantity. By (\ref{l4l2}) and Lemma \ref{lem:l2norm_equality}, for each fixed $y\in\mathbb{R}^{d}$, we have
\[\int_{\mathbb{R}^{d}}\langle x,y\rangle^{4}\,d\mu(x)\leq\beta^{4}\left(\int_{\mathbb{R}^{d}}\langle x,y\rangle^{2}\,d\mu(x)\right)^{2}=\beta^{4}\|B_{\mu}^{1/2}y\|_{2}^{4},\]
and so
\begin{eqnarray*}
\int_{\mathbb{R}^{d}}\int_{\mathbb{R}^{d}}\langle x,y\rangle^{4}\,d\mu(x)\,d\mu(y)&\leq&\beta^{4}\int_{\mathbb{R}^{d}}\|B_{\mu}^{1/2}y\|_{2}^{4}\,d\mu(y)\\&\leq&
\beta^{4}\cdot\beta^{4}\|B_{\mu}^{1/2}B_{\mu}^{1/2}\|_{\mathrm{F}}^{4}=\beta^{8}\|B_{\mu}\|_{\mathrm{F}}^{4},
\end{eqnarray*}
where the second step follows from Lemma \ref{lpl2advance} and the assumption (\ref{l4l2}). Combining this with (\ref{firstmainproofeq2}), we obtain
\[\sum_{i=1}^{d(d+1)/2}\lambda_{i}(T_{\mu})^{2}\leq\beta^{8}\|B_{\mu}\|_{\mathrm{F}}^{4}.\]
This completes the proof of the last inequality in (\ref{firstmaineq2}).
\end{proof}
\section{Proof of the second main result}\label{2ndmainproofsection}
In this section, we prove the second main result Theorem \ref{thm:spectral_decomposition}. The proof has 4 parts. In the first part, we obtain an upper bound for $s(\mu)$ in Proposition \ref{usefulsmubound}. In the second part, we obtain a lower bound for $s(\mu)$ in Proposition \ref{secondmainlbexplicit}. In the third part, we prove an elementary linear algebra result Lemma \ref{lem:T-xtensorx_bound}. In the last part, we complete the proof of Theorem \ref{thm:spectral_decomposition} using the results obtained in the first 3 parts. As a byproduct of this proof, we obtain the following formula for $s(\mu)$ (see Remark \ref{smuformula}):
\[s(\mu)=\sup_{\|A\|_{\mathrm{F}}\leq 1}\inf_{b\in\mathbb{R}}\int_{\mathbb{R}^{d}}|\langle Ax,x\rangle-b|\,d\mu(x),\]
where the supremum is over all $A\in\mathbb{R}_{\mathrm{sym}}^{d\times d}$ with $\|A\|_{\mathrm{F}}\leq 1$.

Throughout this section, $I(\cdot)$ denotes the indicator function of the underlying event, i.e., its value is $1$ when the event occurs and its value is $0$ otherwise. For a real random variable $W$, a median of $W$ is a real number $b_{0}$ that satisfies  
\begin{align*}
\mu(W \geq b_0) &\geq \frac{1}{2} \\
\mu(W \leq b_0) &\geq \frac{1}{2}.
\end{align*}
By definition, if $M \in \mathbb{R}_{\mathrm{sym}}^{d\times d}$, then $M \otimes M$ is the rank one linear operator on $\mathbb{R}_{\mathrm{sym}}^{d\times d}$ defined by $A \mapsto \inner{A}{M}M$ for $A \in \mathbb{R}_{\mathrm{sym}}^{d \times d}$.

\subsection{Upper bound for $s(\mu)$}

\begin{lemma} \label{lem:Tmu-BtensorB}
Let $\mu\in\mathcal{P}_{4}(\mathbb{R}^{d})$. Then the operator $T_\mu - B_{\mu} \otimes B_{\mu}$ on $\mathbb{R}_{\mathrm{sym}}^{d\times d}$ is positive semidefinite and 
\[\langle(T_\mu - B_{\mu} \otimes B_{\mu})(A),A\rangle=\mathbb{E}|\inner{AX}{X} - \mathbb{E}\inner{AX}{X}|^2,\]
for all $A\in\mathbb{R}_{\mathrm{sym}}^{d\times d}$, where $X$ is a random vector in $\mathbb{R}^{d}$ distributed according to $\mu$.
\end{lemma}
\begin{proof}
Since $T_\mu$ and $B_{\mu}\otimes B_{\mu}$ are self-adjoint operators on $\mathbb{R}_{\mathrm{sym}}^{d\times d}$, the operator $T_\mu-B_{\mu}\otimes B_{\mu}$ is also self-adjoint. For all $A \in \mathbb{R}_{\mathrm{sym}}^{d\times d}$, we have
\begin{align*}
\inner{(T_\mu - B_{\mu}\otimes B_{\mu})(A)}{A} &= \inner{T_\mu(A)}{A} - \inner{B_{\mu}}{A}^2 \\
&= \mathbb{E}|\inner{AX}{X}|^2 - |\mathbb{E}\inner{AX}{X}|^2 \\
&=\mathbb{E}|\inner{AX}{X} - \mathbb{E}\inner{AX}{X}|^2\geq 0,
\end{align*}
so $T_\mu - B_{\mu}\otimes B_{\mu}$ is positive semidefinite.
\end{proof}

\begin{prop}\label{usefulsmubound}
Let $\mu\in\mathcal{P}_{4}(\mathbb{R}^{d})$. Then
\[s(\mu)\leq\|T_{\mu}-B_{\mu}\otimes B_{\mu}\|_{\mathrm{op}}^{1/2}.\]
\end{prop}
\begin{proof}
Fix probability measures $\mu_{1},\mu_{2}$ on $\mathbb{R}^{d}$ such that $\mu=\frac{1}{2}\mu_{1}+\frac{1}{2}\mu_{2}$. For all $A\in\mathbb{R}_{\mathrm{sym}}^{d\times d}$ and $b\in\mathbb{R}$, we have
\begin{align*}
&\frac{1}{2}\left|\int_{\mathbb{R}^{d}}\langle A,xx^{T}\rangle\,d\mu_{1}(x)-\int_{\mathbb{R}^{d}}\langle A,xx^{T}\rangle\,d\mu_{2}(x)\right|\\
=&\frac{1}{2}\left|\int_{\mathbb{R}^{d}}\langle A,xx^{T}\rangle-b\,d\mu_{1}(x)-\int_{\mathbb{R}^{d}}\langle A,xx^{T}\rangle-b\,d\mu_{2}(x)\right|\\
\leq&\frac{1}{2}\left(\int_{\mathbb{R}^{d}}|\langle A,xx^{T}\rangle-b|\,d\mu_{1}(x)+\int_{\mathbb{R}^{d}}|\langle A,xx^{T}\rangle-b|\,d\mu_{2}(x)\right)\\
=&\int_{\mathbb{R}^{d}}|\langle A,xx^{T}\rangle-b|\,d\mu(x)=\mathbb{E}|\langle AX,X\rangle-b|,
\end{align*}
where $X$ is a random vector in $\mathbb{R}^{d}$ distributed according to $\mu$. Thus, taking infimum over all $b\in\mathbb{R}$ and then taking supremum over all $\mu_{1},\mu_{2}$ such that $\mu=\frac{1}{2}\mu_{1}+\frac{1}{2}\mu_{2}$, we obtain
\[\frac{1}{2}\cdot\sup_{\mu=\frac{1}{2}\mu_{1}+\frac{1}{2}\mu_{2}}\left|\int_{\mathbb{R}^{d}}\langle A,xx^{T}\rangle\,d\mu_{1}(x)-\int_{\mathbb{R}^{d}}\langle A,xx^{T}\rangle\,d\mu_{2}(x)\right|\leq\inf_{b\in\mathbb{R}}\mathbb{E}|\langle AX,X\rangle-b|.\]
Now, taking supremum over all $A\in\mathbb{R}_{\mathrm{sym}}^{d\times d}$ with $\|A\|_{\mathrm{F}}\leq 1$, we get
\begin{equation}\label{usefulsmuboundproofeq1}
\frac{1}{2}\cdot\sup_{\mu=\frac{1}{2}\mu_{1}+\frac{1}{2}\mu_{2}}\sup_{\|A\|_{\mathrm{F}}\leq1}\left|\int_{\mathbb{R}^{d}}\langle A,xx^{T}\rangle\,d\mu_{1}(x)-\int_{\mathbb{R}^{d}}\langle A,xx^{T}\rangle\,d\mu_{2}(x)\right|\leq
\sup_{\|A\|_{\mathrm{F}}\leq 1}\inf_{b\in\mathbb{R}}\mathbb{E}|\langle AX,X\rangle-b|.
\end{equation}
Since the left hand side is equal to $\displaystyle s(\mu)=\frac{1}{2}\sup_{\mu=\frac{1}{2}\mu_{1}+\frac{1}{2}\mu_{2}}\left\|\int_{\mathbb{R}^{d}}xx^{T}\,d\mu_{1}(x)-\int_{\mathbb{R}^{d}}xx^{T}\,d\mu_{2}(x)\right\|_{\mathrm{F}}$, it follows that
\begin{eqnarray*}
s(\mu)&\leq&\sup_{\|A\|_{\mathrm{F}}\leq 1}\mathbb{E}|\langle AX,X\rangle-\mathbb{E}\langle AX,X\rangle|\\&\leq&
\sup_{\|A\|_{\mathrm{F}}\leq 1}\left(\mathbb{E}|\langle AX,X\rangle-\mathbb{E}\langle AX,X\rangle|^{2}\right)^{1/2}=\|T_{\mu}-B_{\mu}\otimes B_{\mu}\|_{\mathrm{op}}^{1/2},
\end{eqnarray*}
where the last inequality follows from Lemma \ref{lem:Tmu-BtensorB}.
\end{proof}

\subsection{Lower bound for $s(\mu)$}
\begin{lemma} \label{lem:median_optimality}
Let $W$ be a real random variable and let $b_0$ be a median of $W$. Then the following hold:
\begin{itemize}
\item[(i)] $ \displaystyle \mathbb{E}|W - b_0| = \inf_{b \in \mathbb{R}} \mathbb{E}|W-b|$,

\item[(ii)] $ \displaystyle \mathbb{E}|W-\mathbb{E}W|^2 = \inf_{b \in \mathbb{R}}\mathbb{E}|W - b|^2$.
\end{itemize}
\end{lemma}

\begin{lemma} \label{lem:s_F}
Suppose that $\mu$ is a probability measure on a set $\Omega$ and $X$ is a random point in $\Omega$ distributed according to $\mu$. Let $f:\Omega\to\mathbb{R}$ be a function such that $\int_{\Omega}|f(x)|\,d\mu(x)<\infty$. Define the probability measures $\mu_{1},\mu_{2}$ on $\Omega$ by the following densities:
\begin{align}\label{explicitdecomposition}
d\mu_1(x) &=\left(2I(f(x) > b_0) + 2\alpha I(f(x) = b_0)\right)\,d\mu(x) \\
d\mu_2(x) &=\left(2I(f(x) < b_0) + 2(1-\alpha) I(f(x) = b_0)\right)\,d\mu(x),\nonumber
\end{align}
where $b_{0}\in\mathbb{R}$ is a median of the random variable $f(X)$, and $0\leq\alpha\leq 1$ is chosen so that $\mu_{1}(\Omega)=\mu_{2}(\Omega)=1$. Then $\frac{1}{2}\mu_{1}+\frac{1}{2}\mu_{2}=\mu$ and
\[
\frac{1}{2}\left(\int_{\Omega}f(x)\,d\mu_{1}(x)-\int_{\Omega}f(x)\,d\mu_{2}(x)\right) = \inf_{b\in \mathbb{R}} \int_{\Omega} |f(x) - b|\, d\mu(x).
\]
\end{lemma}
\begin{proof}
We begin by noting that there always exist an $0\leq\alpha\leq 1$ such that $\mu_{1}(\Omega)=\mu_{2}(\Omega)=1$, since $b_0$ is a median of $f(X)$, so $\mu(\{x\in \Omega|\,f(x) \geq b_0\}) \geq \frac{1}{2}$ and $\mu(\{x\in \Omega|\,f(x) \leq b_0\}) \geq \frac{1}{2}$.

Since
\[\frac{1}{2}\left(2I(f(x) > b_0) + 2\alpha I(f(x) = b_0)\right)+\frac{1}{2}\left(2I(f(x) < b_0) + 2(1-\alpha) I(f(x) = b_0)\right)=1,\]
we have $\frac{1}{2}\mu_{1}+\frac{1}{2}\mu_{2}=\mu$. Next,
\begin{align*}
 &\frac{1}{2} \left(\int_{\Omega} f(x)\,d\mu_{1}(x)-\int_\Omega f(x)\,d\mu_{2}(x)\right) \\
=& \,\frac{1}{2} \left(\int_{\Omega} (f(x)-b_0)\,d\mu_{1}(x)-\int_\Omega (f(x)-b_0)\,d\mu_{2}(x)\right) \\
=& \,\frac{1}{2} \left(\int_{\Omega} 2(f(x)-b_0)I(f(x) > b_0) \,d\mu(x)-\int_\Omega 2(f(x)-b_0)I(f(x) < b_0)\,d\mu(x)\right) \\
= &\,\int_\Omega |f(x) - b_0|\,d\mu(x) \\
= &\,\inf_{b\in \mathbb{R}} \int_{\Omega} |f(x) - b|\, d\mu(x),
\end{align*}
where the last equality follows from Lemma \ref{lem:median_optimality}(i).
\end{proof}


\begin{lemma} \label{lem:mean_lower_bound}
Let $U$ be a real random variable with $U \geq 0$ then 
\[\mathbb{E}{U}\geq\frac{(\mathbb{E}U^2)^{3/2}}{(\mathbb{E}U^4)^{1/2}}. \]
\end{lemma}
\begin{proof}
By Holder's inequality,
\[ \mathbb{E}{U^2} = \mathbb{E}(U^{2/3}U^{4/3}) \leq (\mathbb{E}U)^{2/3}(\mathbb{E}U^4)^{1/3}. \]
\end{proof}

\begin{lemma} \label{lem:inf_of_mean_bounds}
Let $W$ be a real random variable. Then 
\[ 
\inf_{b \in \mathbb{R}} \mathbb{E}|W-b|\geq\frac{(\mathbb{E}|W-\mathbb{E}W|^2)^{3/2}}{5(\mathbb{E}W^4)^{1/2}}.
\]
\end{lemma}
\begin{proof}
Let $b_0$ be a median of $W$. Then 
\begin{equation} \label{eq:full_fraction}
\inf_{b \in \mathbb{R}} \mathbb{E}|W-b| = \mathbb{E}|W - b_0| \geq \frac{(\mathbb{E}|W - b_0|^2)^{3/2}}{(\mathbb{E}|W - b_0|^4)^{1/2}},
\end{equation}
where the equality follows from Lemma \ref{lem:median_optimality}(i) and the inequality follows from Lemma \ref{lem:mean_lower_bound}. 
By Lemma \ref{lem:median_optimality}(ii), we have
\begin{equation} \label{eq:numerator_upper_bound}
(\mathbb{E}|W-\mathbb{E}W|^2)^{3/2} \leq (\mathbb{E}|W-b_0|^2)^{3/2}.
\end{equation} 
To get an upper bound on the denominator $(\mathbb{E}|W - b_0|^4)^{1/2}$ in (\ref{eq:full_fraction}), notice that $b_0 \leq 2^{1/4}(\mathbb{E}W^4)^{1/4}$
since 
\[ \mathbb{E}W^4 \geq \mathbb{E}b_0^4I(W\geq b_0) = b_0^4\cdot\mathbb{P}(W\geq b_0)\geq \tfrac{1}{2}b_0^4
\] 
where $I(\cdot)$ is the indicator random variable for the underlying event. Hence,
\begin{align*}
(\mathbb{E}|W-b_0|^4)^{1/4} 
&\leq (\mathbb{E}(|W|+b_0)^4)^{1/4} \\
&\leq (\mathbb{E}W^4)^{1/4} + b_0 \\
&\leq (1+2^{1/4})(\mathbb{E}W^4)^{1/4},
\end{align*}
where the second inequality follows from Minkowski's inequality. Thus we have
\begin{equation}\label{eq:denominator_upper_bound}
(\mathbb{E}|W - b_0|^4)^{1/2} \leq (1+2^{1/4})^2(\mathbb{E}W^4)^{1/2} \leq 5(\mathbb{E}W^4)^{1/2}.
\end{equation}
Applying (\ref{eq:numerator_upper_bound}) to the numerator and  (\ref{eq:denominator_upper_bound}) to the denominator in (\ref{eq:full_fraction}) completes the proof.
\end{proof}

\begin{lemma}\label{lem:B_lower_bound}
Let $\mu\in\mathcal{P}_{4}(\mathbb{R}^{d})$. Suppose that there exists $\beta\geq 1$ satisfying (\ref{l8l2}) for all $v\in\mathbb{R}^{d}$. Then 
\[
\left(\int_{\mathbb{R}^d}\langle Ax,x\rangle^{4}\,d\mu(x)\right)^{1/4} \leq \beta^2\|B_{\mu}\|_{\mathrm{F}}
\]
for all $A\in\mathbb{R}_{\mathrm{sym}}^{d\times d}$ with $\|A\|_\mathrm{F} \leq 1$.
\end{lemma}
\begin{proof}
Since $A$ is symmetric, we can use its spectral decomposition to write $A=A_{+}-A_{-}$ so that $A_{+}$ and $A_{-}$ are positive semidefinite and $\|A_{+}+A_{-}\|_{\mathrm{F}}=\|A\|_{\mathrm{F}}$. We have
\begin{eqnarray*}
\left(\int_{\mathbb{R}^d}\langle Ax,x\rangle^{4}\,d\mu(x)\right)^{1/4}&=&
\left(\int_{\mathbb{R}^d}(\langle A_{+}x,x\rangle-\langle A_{-}x,x\rangle)^{4}\,d\mu(x)\right)^{1/4}\\&\leq&
\left(\int_{\mathbb{R}^d}\langle A_{+}x,x\rangle^{4}\,d\mu(x)\right)^{1/4}+\left(\int_{\mathbb{R}^d}\langle A_{-}x,x\rangle^{4}\,d\mu(x)\right)^{1/4}\\&=&
\left(\int_{\mathbb{R}^d}\|A_{+}^{1/2}x\|_{2}^{8}\,d\mu(x)\right)^{1/4}+\left(\int_{\mathbb{R}^d}\|A_{-}^{1/2}x\|_{2}^{8}\,d\mu(x)\right)^{1/4}\\&\leq&
\beta^{2}\|B_{\mu}^{1/2}A_{+}^{1/2}\|_{\mathrm{F}}^{2}+\beta^{2}\|B_{\mu}^{1/2}A_{-}^{1/2}\|_{\mathrm{F}}^{2}\\&=&
\beta^{2}(\mathrm{Tr}(B_{\mu}A_{+})+\mathrm{Tr}(B_{\mu}A_{-}))\\&=&
\beta^{2}\mathrm{Tr}(B_{\mu}(A_{+}+A_{-}))\\&\leq&
\beta^{2}\|B_{\mu}\|_{\mathrm{F}}\|A_{+}+A_{-}\|_{F}\\&=&
\beta^{2}\|B_{\mu}\|_{F}\|A\|_{F},
\end{eqnarray*}
where the second step follows from Minkowski's inequality and the fourth step follows from Lemma \ref{lpl2advance} and the assumption (\ref{l8l2}).
\end{proof}

Recall from Lemma \ref{lem:Tmu-BtensorB} that the operator $T_{\mu}-B_{\mu}\otimes B_{\mu}$ on $\mathbb{R}_{\mathrm{sym}}^{d\times d}$ is always positive semidefinite.

\begin{prop}\label{secondmainlbexplicit}
Suppose that $\mu\in\mathcal{P}_{8}(\mathbb{R}^{d})$, $\mu\neq\delta_{0}$ and $\beta\geq 1$ satisfies (\ref{l8l2}) for all $v\in\mathbb{R}^{d}$. Let $A\in\mathbb{R}_{\mathrm{sym}}^{d\times d}$ be the normalized leading eigenvector of $T_{\mu}-B_{\mu}\otimes B_{\mu}$, i.e., $\|A\|_{\mathrm{F}}=1$ and $(T_{\mu}-B_{\mu}\otimes B_{\mu})(A)=[\lambda_{1}(T_{\mu}-B_{\mu}\otimes B_{\mu})]A$. Define the probability measures $\mu_{1}$ and $\mu_{2}$ on $\mathbb{R}^{d}$ as follows:
\begin{align*}
d\mu_1(x) &= 2I(\langle Ax,x\rangle > b_0)\,d\mu(x) + 2\alpha I(\langle Ax,x\rangle = b_0)\,d\mu(x),\\
d\mu_2(x) &= 2I(\langle Ax,x\rangle < b_0)\,d\mu(x) + 2(1-\alpha) I(\langle Ax,x\rangle = b_0)\,d\mu(x),\nonumber
\end{align*}
where $0\leq\alpha\leq 1$ is chosen so that $\mu_{1}(\mathbb{R}^{d})=\mu_{2}(\mathbb{R}^{d})=1$, and $b_{0}$ is a median of the random variable $\langle AX,X\rangle$, where $X$ is a random vector in $\mathbb{R}^{d}$ distributed according to $\mu$. Then $\mu=\frac{1}{2}\mu_{1}+\frac{1}{2}\mu_{2}$ and
\[s(\mu)\geq\frac{1}{2}\|B_{\mu_{1}}-B_{\mu_{2}}\|_{\mathrm{F}}\geq\frac{\|T_{\mu}-B_{\mu}\otimes B_{\mu}\|_{\mathrm{op}}^{3/2}}{5\beta^{4}\|B_{\mu}\|_{\mathrm{F}}^{2}}.\]
\end{prop}
\begin{proof}
Taking $\Omega=\mathbb{R}^{d}$ and $f(x)=\langle Ax,x\rangle$, for $x\in\mathbb{R}^{d}$, in Lemma \ref{lem:s_F}, we obtain
\[\frac{1}{2}\left(\int_{\mathbb{R}^{d}}\langle Ax,x\rangle\,d\mu_{1}(x)-\int_{\mathbb{R}^{d}}\langle Ax,x\rangle\,d\mu_{2}(x)\right)=\inf_{b\in\mathbb{R}}\int_{\mathbb{R}^{d}}|\langle Ax,x\rangle-b|\,d\mu(x).\]
Since the left hand side is equal to $\displaystyle\frac{1}{2}\left\langle A,\int_{\mathbb{R}^{d}}xx^{T}\,d\mu_{1}(x)-\int_{\mathbb{R}^{d}}xx^{T}\,d\mu_{2}(x)\right\rangle$ and since $\|A\|_{\mathrm{F}}=1$, this implies that
\begin{equation}\label{secondmainlbexplicitproofeq1}
\frac{1}{2}\left\|\int_{\mathbb{R}^{d}}xx^{T}\,d\mu_{1}(x)-\int_{\mathbb{R}^{d}}xx^{T}\,d\mu_{2}(x)\right\|_{\mathrm{F}}\geq\inf_{b\in\mathbb{R}}\int_{\mathbb{R}^{d}}|\langle Ax,x\rangle-b|\,d\mu(x).
\end{equation}
But by Lemma \ref{lem:inf_of_mean_bounds},
\[\inf_{b\in\mathbb{R}}\int_{\mathbb{R}^{d}}|\langle Ax,x\rangle-b|\,d\mu(x)=\inf_{b\in\mathbb{R}}\mathbb{E}|\langle AX,X\rangle-b|\geq\frac{(\mathbb{E}|\langle AX,X\rangle-\mathbb{E}\langle AX,X\rangle|^{2})^{3/2}}{5\left(\mathbb{E}\langle AX,X\rangle^{4}\right)^{1/2}}.\]
By Lemma \ref{lem:Tmu-BtensorB},
\[\mathbb{E}|\langle AX,X\rangle-\mathbb{E}\langle AX,X\rangle|^{2}=\langle(T_{\mu}-B_{\mu}\otimes B_{\mu})(A),A\rangle=\|T_{\mu}-B_{\mu}\otimes B_{\mu}\|_{\mathrm{op}},\]
and by Lemma \ref{lem:B_lower_bound},
\[\mathbb{E}\langle AX,X\rangle^{4}\leq\beta^{8}\|B_{\mu}\|_{\mathrm{F}}^{4}.\]
Therefore,
\[\inf_{b\in\mathbb{R}}\int_{\mathbb{R}^{d}}|\langle Ax,x\rangle-b|\,d\mu(x)\geq\frac{\|T_{\mu}-B_{\mu}\otimes B_{\mu}\|_{\mathrm{op}}^{3/2}}{5\beta^{4}\|B_{\mu}\|_{\mathrm{F}}^{2}}.\]
So by (\ref{secondmainlbexplicitproofeq1}),
\[\frac{1}{2}\|B_{\mu_{1}}-B_{\mu_{2}}\|_{\mathrm{F}}=\frac{1}{2}\left\|\int_{\mathbb{R}^{d}}xx^{T}\,d\mu_{1}(x)-\int_{\mathbb{R}^{d}}xx^{T}\,d\mu_{2}(x)\right\|_{\mathrm{F}}\geq\frac{\|T_{\mu}-B_{\mu}\otimes B_{\mu}\|_{\mathrm{op}}^{3/2}}{5\beta^{4}\|B_{\mu}\|_{\mathrm{F}}^{2}}\]
\end{proof}

\begin{remark}\label{smuformula}
From (\ref{usefulsmuboundproofeq1}) and (\ref{secondmainlbexplicitproofeq1}) (note that (\ref{secondmainlbexplicitproofeq1}) actually holds for all $A\in\mathbb{R}_{\mathrm{sym}}^{d\times d}$ with $\|A\|_{\mathrm{F}}\leq 1$), we have
\[s(\mu)=\sup_{\|A\|_{\mathrm{F}}\leq 1}\inf_{b\in\mathbb{R}}\int_{\mathbb{R}^{d}}|\langle Ax,x\rangle-b|\,d\mu(x),\]
where the supremum is over all $A\in\mathbb{R}_{\mathrm{sym}}^{d\times d}$ with $\|A\|_{\mathrm{F}}\leq 1$.
\end{remark}

\subsection{Some linear algebra}
\begin{lemma}\cite[p. 157]{Audenart} \label{lem:audenart}
Let $V$ be a finite dimensional inner product space. If $T:V \to V$ is a positive semidefinite operator and $P: V \to V$ is an orthogonal projection, then 
\[
\opnorm{T} \leq \opnorm{PTP} + \opnorm{(I-P)T(I-P)}
\]
\end{lemma}

\begin{lemma} \label{lem:psd_proj}
Let $A, B\in\mathbb{R}^{d\times d}$ be such that $A$ and $B-A$ are positive semidefinite. Then for every orthogonal projection $Q$ on $\mathbb{R}^{d}$, we have $\opnorm{QAQ} \leq \opnorm{QBQ}$.
\end{lemma}
\begin{proof}
Note
$\|QAQ\|_{\mathrm{op}}=\sup_{\|x\|_{2}\leq 1}\langle QAQx,x\rangle\leq\sup_{\|x\|_{2}\leq 1}\langle QBQx,x\rangle=\|QBQ\|_{\mathrm{op}}.$
\end{proof}

\begin{lemma} \label{lem:T-xtensorx_bound} 
Suppose that $V$ is a finite dimensional inner product space, $T:V \to V$ is linear and $x_0 \in V$. If $T-x_0 \otimes x_0$ is positive semidefinite, then 
\[
\frac{1}{2}(\lambda_2(T) + \lambda_1(T) - \|x_0\|^2)  \leq \opnorm{T-x_0 \otimes x_0} \leq 2(\lambda_2(T) + \lambda_1(T) - \|x_0\|^2).
\]
\end{lemma}
\begin{proof} Let $T_0 = T-x_0 \otimes x_0$. 
Since $T$ is positive semidefinite, $\lambda_1(T) = \opnorm{T}$ and  
\[\lambda_2(T) = \inf_{R} \opnorm{T-R},\] 
where the infimum is over all rank 1 linear maps $R:V\to V$. Thus 
\begin{equation} \label{eq:psd_proj_upper_bound_1}
\opnorm{T-x_{0}\otimes x_{0}} \geq \lambda_2(T)
\end{equation}
and 
\begin{equation}\label{eq:psd_proj_upper_bound_2}
\opnorm{T-x_0 \otimes x_0} \geq \|T\|_{\mathrm{op}} - \opnorm{x_0 \otimes x_0} = \lambda_1(T) - \|x_0\|^2_2.
\end{equation}
The resulting lower bound can be concluded by taking the average of (\ref{eq:psd_proj_upper_bound_1}) and (\ref{eq:psd_proj_upper_bound_2}).

To get the upper bound, let $y \in V$ be a normalized leading eigenvector of $T$, i.e., $Ty = \lambda_1(T)y$ and $\|y\| = 1$. Then $P := y \otimes y$ is an orthogonal projection and by Lemma \ref{lem:audenart}, 
\begin{equation} \label{eq:psd_proj_lower_bound_1}
\opnorm{T_0} \leq \opnorm{PT_0P} + \opnorm{(I-P)T_0(I-P)}
\end{equation}

Since $T_{0}$ and $T-T_{0}$ are positive semidefinite, by Lemma \ref{lem:psd_proj}, we have
\[\opnorm{(I-P)T_0(I-P)}\leq\opnorm{(I-P)T(I-P)} = \lambda_2(T).\]
Hence, by (\ref{eq:psd_proj_lower_bound_1}),
\[
\opnorm{T_0} \leq \opnorm{PT_0P} + \lambda_2(T) = \langle T_0y, y \rangle + \lambda_2(T),
\]
and using $\langle T_0y, y \rangle = \langle Ty, y \rangle - \langle x_0, y \rangle^2 = \lambda_1(T) - \langle x_0, y \rangle^2$, we get
\begin{align}\label{eq:bound_on_T0}
\opnorm{T_0} 
&\leq \lambda_2(T) + \lambda_1(T) - \langle x_0, y \rangle^2 \\
\nonumber &= \lambda_2(T) + \lambda_1(T) - \lnorm{x_0}{2}^2 + (\lnorm{x_0}{2}^2 - \inner{x_0}{y}^2). 
\end{align}

Notice that by the Pythegorean Theorem,
\[\|x_0\|^2 - \inner{x_0}{y}^2 = \|x_0\|^2 - \|Px_0\|^2 
= \|(I-P)x_0\|^2,\]
and by Lemma \ref{lem:psd_proj},
\begin{equation*}
\|(I-P)x_0\|^2 = \opnorm{(I-P)(x_0\otimes x_0)(I-P)} \leq \opnorm{(I-P)T(I-P)} = \lambda_2(T).
\end{equation*} 
Therefore,
\begin{equation}\label{eq:bound_on_l2}
\|x_0\|^2 - \inner{x_0}{y}^2 \leq \lambda_2(T).
\end{equation}

Using (\ref{eq:bound_on_l2}) in (\ref{eq:bound_on_T0}) gives
\begin{align*}
\opnorm{T_0} &\leq 2\lambda_2(T) + \lambda_1(T) - \|x_{0}\|^2\\&\leq
2(\lambda_2(T)+\lambda_1(T)-\|x_{0}\|^2),
\end{align*}
where the last inequality follows from the fact that $T-x_{0}\otimes x_{0}$ is positive semidefinite, so $\lambda_{1}(T)=\|T\|_{\mathrm{op}}\geq\|x_{0}\|^{2}$.
\end{proof}

\subsection{Completing the proof}

\begin{proof}[Proof Theorem \ref{thm:spectral_decomposition}]
By Proposition \ref{usefulsmubound} and Proposition \ref{secondmainlbexplicit}, we have
\[\frac{\|T_{\mu}-B_{\mu}\otimes B_{\mu}\|_{\mathrm{op}}^{3}}{25\beta^{8}\|B_{\mu}\|_{\mathrm{F}}^{4}}\leq s(\mu)^{2}\leq\|T_{\mu}-B_{\mu}\otimes B_{\mu}\|_{\mathrm{op}}.\]
Since the operator $T_{\mu}-B_{\mu}\otimes B_{\mu}$ is positive semidefinite (by Lemma \ref{lem:Tmu-BtensorB}), we can apply Lemma \ref{lem:T-xtensorx_bound} and obtain
\[\frac{1}{2}(\lambda_{2}(T_{\mu})+\lambda_{1}(T_{\mu})-\|B_{\mu}\|_{\mathrm{F}}^{2})\leq \|T_{\mu}-B_{\mu}\otimes B_{\mu}\|_{\mathrm{op}}\leq 2(\lambda_{2}(T_{\mu})+\lambda_{1}(T_{\mu})-\|B_{\mu}\|_{\mathrm{F}}^{2}).\]
Therefore,
\[\frac{(\lambda_{2}(T_{\mu})+\lambda_{1}(T_{\mu})-\|B_{\mu}\|_{\mathrm{F}}^{2})^{3}}{200\beta^{8}\|B_{\mu}\|_{\mathrm{F}}^{4}}\leq s(\mu)^{2}\leq2(\lambda_{2}(T_{\mu})+\lambda_{1}(T_{\mu})-\|B_{\mu}\|_{\mathrm{F}}^{2}),\]
and so dividing by $\|B_{\mu}\|_{\mathrm{F}}^{2}$, we get
\begin{equation}\label{secondmainproofeq1}
\frac{1}{200\beta^{8}}\left(\frac{\lambda_{2}(T_{\mu})+\lambda_{1}(T_{\mu})-\|B_{\mu}\|_{\mathrm{F}}^{2}}{\|B_{\mu}\|_{\mathrm{F}}^{2}}\right)^{3}\leq \left(\frac{s(\mu)}{\|B_{\mu}\|_{\mathrm{F}}}\right)^{2}\leq\frac{2(\lambda_{2}(T_{\mu})+\lambda_{1}(T_{\mu})-\|B_{\mu}\|_{\mathrm{F}}^{2})}{\|B_{\mu}\|_{\mathrm{F}}^{2}}.
\end{equation}
Since $\|B_{\mu}\|_{\mathrm{F}}^{2}\leq\lambda_{1}(T_{\mu})$ by Theorem \ref{firstmain}, the lower bound for $(\frac{s(\mu)}{\|B_{\mu}\|_{\mathrm{F}}})^{2}$ in (\ref{secondmaineq1}) follows. To prove the upper bound, we consider two cases.

Case 1: $\lambda_1(T_{\mu}) \leq 2 \fnorm{B_{\mu}}^2$.

We have 
\begin{align*}
\left(\frac{s(\mu)}{\|B_{\mu}\|_{\mathrm{F}}}\right)^{2}\leq
\frac{2(\lambda_{2}(T_{\mu})+\lambda_{1}(T_{\mu})-\|B_{\mu}\|_{\mathrm{F}}^{2})}{\|B_{\mu}\|_{\mathrm{F}}^{2}}\leq
\frac{2(\lambda_{2}(T_{\mu})+\lambda_{1}(T_{\mu})-\|B_{\mu}\|_{\mathrm{F}}^{2})}{\frac{1}{2}\lambda_{1}(T_{\mu})}.
\end{align*}
Thus, the upper bound for $(\frac{s(\mu)}{\|B_{\mu}\|_{\mathrm{F}}})^{2}$ in (\ref{secondmaineq1}) follows in this case.

Case 2: $\lambda_1(T_{\mu}) \geq 2 \fnorm{B_{\mu}}^2$.

We have
\[4\left[\frac{\lambda_{2}(T_{\mu})}{\lambda_{1}(T_{\mu})}+\left(1-\frac{\|B_{\mu}\|_{\mathrm{F}}^{2}}{\lambda_{1}(T_{\mu})}\right)\right]\geq
4\left(1-\frac{\|B_{\mu}\|_{\mathrm{F}}^{2}}{\lambda_{1}(T_{\mu})}\right)\geq 4\left(1-\frac{1}{2}\right)=2.\]
On the other hand, by (\ref{eq:s_remark}), we have $(\frac{s(\mu)}{\fnorm{B}})^2 \leq 1$. Thus, the upper bound for $(\frac{s(\mu)}{\|B_{\mu}\|_{\mathrm{F}}})^{2}$ in (\ref{secondmaineq1}) also holds in this case.
\end{proof}

\section{Proof of the examples I}\label{smuexamplesproofs}
In this section, we prove the bounds for $s(\mu)$ stated in Example \ref{introproduct}-Example \ref{introgeneralmixture}. Throughout this section, $[d]:=\{1,\ldots,d\}$ for $d\in\mathbb{N}$.
\begin{lemma}\label{cumulant}
Suppose that $X_{1},\ldots,X_{d}$ are independent identically distributed random variables with $\mathbb{E}X_{1}=0$, $\mathbb{E}X_{1}^{2}=1$ and $\mathbb{E}X_{1}^{4}<\infty$. Then for all $i,j,k,\ell\in[d]$,
\begin{equation}\label{cumulanteq}
\mathbb{E}(X_{i}X_{j}X_{k}X_{\ell})=\delta_{i=j}\delta_{k=\ell}+\delta_{i=k}\delta_{j=\ell}+\delta_{i=\ell}\delta_{j=k}+(\mathbb{E}X_{1}^{4}-3)\delta_{i=j=k=\ell},
\end{equation}
where $\delta_{i=j}=\begin{cases}1,&i=j\\0,&i\neq j\end{cases}$ and $\delta_{i=j=k=\ell}=\begin{cases}1,&i=j=k=\ell\\0,&\text{Otherwise}\end{cases}$.
\end{lemma}
\begin{proof}
Case 1: One of $i,j,k,\ell$ is different from the rest, e.g., $i\neq j,k,\ell$.

Both sides of (\ref{cumulanteq}) are equal to $0$.

Case 2: $(i=j\neq k=\ell)$ or $(i=k\neq j=\ell)$ or $(i=\ell\neq j=k)$.

Both sides of (\ref{cumulanteq}) are equal to $1$.

Case 3: $i=j=k=\ell$.

Both sides of (\ref{cumulanteq}) are equal to $\mathbb{E}X_{1}^{4}$.
\end{proof}

\begin{lemma}\label{tmuiid}
Suppose that $X_{1},\ldots,X_{d}$ are independent identically, distributed random variables with $\mathbb{E}X_{1}=0$, $\mathbb{E}X_{1}^{2}=1$ and $\mathbb{E}X_{1}^{4}<\infty$. Let $\mu$ be the distribution of the random vector $X=(X_{1},\ldots,X_{d})^{T}$. Then
\[T_{\mu}(A)=2A+\mathrm{Tr}(A)I+(\mathbb{E}X_{1}^{4}-3)\mathrm{diag}(A),\]
for all $A\in\mathbb{R}_{\mathrm{sym}}^{d\times d}$, where $\mathrm{diag}(A)$ is the $d\times d$ diagonal matrix with diagonal entries being the diagonal entries of $A$.
\end{lemma}
\begin{proof}
The inner product $\langle\cdot\,,\cdot\rangle$ defined in Definition \ref{def:T_mu} can be extended to the vector space $\mathbb{R}^{d\times d}$ by setting $\langle A,B\rangle:=\mathrm{Tr}(AB^{T})$ for $A,B\in\mathbb{R}^{d\times d}$.

Fix $A\in\mathbb{R}_{\mathrm{sym}}^{d\times d}$. We have
\[T_{\mu}(A)=\mathbb{E}\langle A,XX^{T}\rangle XX^{T}=\sum_{i,j,k,\ell\in[d]}(\mathbb{E}X_{i}X_{j}X_{k}X_{\ell})\langle A,e_{i}e_{j}^{T}\rangle e_{k}e_{\ell}^{T}.\]
But by Lemma \ref{cumulant},
\[\mathbb{E}(X_{i}X_{j}X_{k}X_{\ell})=\delta_{i=j}\delta_{k=\ell}+\delta_{i=k}\delta_{j=\ell}+\delta_{i=\ell}\delta_{j=k}+(\mathbb{E}X_{1}^{4}-3)\delta_{i=j=k=\ell}.\]
Therefore,
\begin{align*}
T_{\mu}(A)=&
\sum_{i,k\in[d]}\langle A,e_{i}e_{i}^{T}\rangle e_{k}e_{k}^{T}+\sum_{i,j\in[d]}\langle A,e_{i}e_{j}^{T}\rangle e_{i}e_{j}^{T}+\sum_{i,j\in[d]}\langle A,e_{i}e_{j}^{T}\rangle e_{j}e_{i}^{T}\\&
-(\mathbb{E}X_{1}^{4}-3)\sum_{i\in[d]}\langle A,e_{i}e_{i}^{T}\rangle e_{i}e_{i}^{T}\\=&
\mathrm{Tr}(A)I+A+A^{T}-(\mathbb{E}X_{1}^{4}-3)\mathrm{diag}(A).
\end{align*}
Since $A^{T}=A$, the result follows.
\end{proof}

\begin{lemma}\label{tmugaussian}
Let $\Sigma$ be a $d\times d$ positive semidefinite matrix and $\mu=\mathcal{N}(0,\Sigma)$. Then
\[T_{\mu}(A)=2\Sigma A\Sigma+\langle A,\Sigma\rangle\Sigma,\]
for all $A\in\mathbb{R}_{\mathrm{sym}}^{d\times d}$.
\end{lemma}
\begin{proof}
Let $g\sim\mathcal{N}(0,I)$ be a standard Gaussian vector in $\mathbb{R}^{d}$. By Lemma \ref{tmuiid}, we have
\begin{equation}\label{tmugaussianproofeq1}
\mathbb{E}\langle Mg,g\rangle gg^{T}=2M+\mathrm{Tr}(M)I,
\end{equation}
for all $M\in\mathbb{R}_{\mathrm{sym}}^{d\times d}$. Let $X=\Sigma^{\frac{1}{2}}g$. Since $X$ has distribution $\mathcal{N}(0,\Sigma)$, for all $A\in\mathbb{R}_{\mathrm{sym}}^{d\times d}$, we have
\begin{eqnarray*}
T_{\mu}(A)&=&\mathbb{E}\langle A,XX^{T}\rangle XX^{T}\\&=&
\mathbb{E}\langle AX,X\rangle XX^{T}\\&=&
\mathbb{E}\langle\Sigma^{\frac{1}{2}}A\Sigma^{\frac{1}{2}}g,g\rangle\Sigma^{\frac{1}{2}}gg^{T}\Sigma^{\frac{1}{2}}\\&=&
\Sigma^{\frac{1}{2}}\left(\mathbb{E}\langle\Sigma^{\frac{1}{2}}A\Sigma^{\frac{1}{2}}g,g\rangle gg^{T}\right)\Sigma^{\frac{1}{2}}.
\end{eqnarray*}
Since $A$ is symmetric, $\Sigma^{\frac{1}{2}}A\Sigma^{\frac{1}{2}}$ is also symmetric.  So by (\ref{tmugaussianproofeq1}),
\[T_{\mu}(A)=\Sigma^{\frac{1}{2}}\left(2\Sigma^{\frac{1}{2}}A\Sigma^{\frac{1}{2}}+\mathrm{Tr}(\Sigma^{\frac{1}{2}}A\Sigma^{\frac{1}{2}})I\right)\Sigma^{\frac{1}{2}}=2\Sigma A\Sigma+\mathrm{Tr}(A\Sigma)\Sigma,\]
which proves the lemma.
\end{proof}

\begin{lemma}\label{smu1dsinglenormal}
Let $a\geq 0$. Consider the normal distribution $\mathcal{N}(0,a)$ on $\mathbb{R}$ with mean $0$ and variance $a$. We have $s(\mathcal{N}(0,a))\geq 0.8a$.
\end{lemma}
\begin{proof}
Without loss of generality, since $\mathcal{N}(0,a)$ is the pushforward measure of $\mathcal{N}(0,1)$ by the map $x\mapsto\sqrt{a}\cdot x$, we may assume that $a=1$. Let $f(x)=\frac{1}{\sqrt{2\pi}}e^{-x^{2}/2}$ be the density of $\mathcal{N}(0,1)$. Let $c=0.67448\ldots$ be such that $\int_{-c}^{c}f(x)\,dx=\frac{1}{2}$. Define the continuous probability measures $\mu_{1}$ and $\mu_{2}$ on $\mathbb{R}$ as follows:
\begin{align*}
d\mu_1(x) &= 2f(x)I(-c\leq x\leq c)\,dx\\
d\mu_2(x) &= 2f(x)I(|x|>c)\,dx.\nonumber
\end{align*}
Note that $\frac{1}{2}d\mu_1(x)+\frac{1}{2}d\mu_{2}(x)=f(x)\,dx$, so $\frac{1}{2}\mu_{1}+\frac{1}{2}\mu_{2}=\mathcal{N}(0,1)$. Thus, we have
\begin{eqnarray*}
s(\mu)&\geq&\frac{1}{2}\left(\int_{\mathbb{R}}x^{2}\,d\mu_{2}(x)-\int_{\mathbb{R}}x^{2}\,d\mu_{1}(x)\right)\\&=&
\int_{|x|>c}x^{2}f(x)\,dx-\int_{-c}^{c}x^{2}f(x)\,dx=0.8573\ldots.
\end{eqnarray*}
The result follows.
\end{proof}

\begin{lemma}\label{smuproduct}
Suppose that $\mu$ is a probability measure on $\mathbb{R}^{d_{1}}$ and $\nu$ is a probability measure on $\mathbb{R}^{d_{2}}$. Consider the product measure $\mu\times\nu$ on $\mathbb{R}^{d_{1}+d_{2}}$. We have $s(\mu\times\nu)\geq s(\mu)$.
\end{lemma}
\begin{proof}
Let $P$ be the canonical projection from $\mathbb{R}^{d_{1}+d_{2}}=\mathbb{R}^{d_{1}}\times\mathbb{R}^{d_{2}}$ onto $\mathbb{R}^{d_{1}}$. Let $\mu_{1},\mu_{2}$ be probability measures on $\mathbb{R}^{d_{1}}$ such that $\mu=\frac{1}{2}\mu_{1}+\frac{1}{2}\mu_{2}$. Then
\[\mu\times\nu=\frac{1}{2}(\mu_{1}\times\nu)+\frac{1}{2}(\mu_{2}\times\nu).\]
Thus, we have
\[s(\mu\times\nu)\geq\frac{1}{2}\|B_{\mu_{1}\times\nu}-B_{\mu_{2}\times\nu}\|_{\mathrm{F}}\geq
\frac{1}{2}\|P(B_{\mu_{1}\times\nu}-B_{\mu_{2}\times\nu})P\|_{\mathrm{F}}=
\frac{1}{2}\|B_{\mu_{1}}-B_{\mu_{2}}\|_{\mathrm{F}}.\]
The result follows by taking supremum over all $\mu_{1},\mu_{2}$ such that $\mu=\frac{1}{2}\mu_{1}+\frac{1}{2}\mu_{2}$.
\end{proof}

\begin{lemma}\label{smumixture}
Suppose that $\mu,\nu$ are a probability measure on $\mathbb{R}$. Consider the mixture measure $\frac{1}{2}\mu+\frac{1}{2}\nu$. We have $s(\frac{1}{2}\mu+\frac{1}{2}\nu)\geq\frac{1}{2}s(\mu)$.
\end{lemma}
\begin{proof}
Fix probability measures $\mu_{1},\mu_{2}$ on $\mathbb{R}^{d}$ such that $\mu=\frac{1}{2}\mu_{1}+\frac{1}{2}\mu_{2}$. Then
\[\frac{1}{2}\mu+\frac{1}{2}\nu=\frac{1}{2}\left(\frac{1}{2}\mu_{1}+\frac{1}{2}\nu\right)+\frac{1}{2}\left(\frac{1}{2}\mu_{2}+\frac{1}{2}\nu\right),\]
so
\[s\left(\frac{1}{2}\mu+\frac{1}{2}\nu\right)\geq\left\|B_{\frac{1}{2}\mu_{1}+\frac{1}{2}\nu}-B_{\frac{1}{2}\mu_{2}+\frac{1}{2}\nu}\right\|_{\mathrm{F}}=\left\|\frac{1}{2}(B_{\mu_{1}}-B_{\mu_{2}})\right\|_{\mathrm{F}}.\]
The result follows by taking supremum over all $\mu_{1},\mu_{2}$ such that $\mu=\frac{1}{2}\mu_{1}+\frac{1}{2}\mu_{2}$.
\end{proof}

\begin{proof}[Proof of the bound for $s(\mu^{(d)})$ in Example \ref{introproduct}]
Fix the dimension $d\in\mathbb{N}$, and for simplicity, let $\mu=\mu^{(d)}$. It is easy to see that $B_{\mu}=I$. By Lemma \ref{tmuiid}, for every $A\in\mathbb{R}_{\mathrm{sym}}^{d\times d}$,
\[(T_{\mu}-B_{\mu}\otimes B_{\mu})(A)=(T_{\mu}-I\otimes I)(A)=T_{\mu}(A)-\mathrm{Tr}(A)I=2A+(\mathbb{E}X_{1}^{4}-3)\mathrm{diag}(A),\]
and so
\[\|(T_{\mu}-B_{\mu}\otimes B_{\mu})(A)\|_{\mathrm{F}}\leq2\|A\|_{\mathrm{F}}+(\mathbb{E}X_{1}^{4}-3)\|A\|_{\mathrm{F}}=(\mathbb{E}X_{1}^{4}+1)\|A\|_{\mathrm{F}}.\]
Hence, $\|T_{\mu}-B_{\mu}\otimes B_{\mu}\|_{\mathrm{op}}\leq\mathbb{E}X_{1}^{4}+1$, so by Proposition \ref{usefulsmubound}, it follows that $s(\mu)\leq(\mathbb{E}X_{1}^{4}+1)^{1/2}$.
\end{proof}

\begin{proof}[Proof of the bounds for $s(\mu)$ in Example \ref{introsinglenormal}]
We first prove the upper bound for $s(\mu)$. Since $B_{\mu}=\Sigma$, by Lemma \ref{tmugaussian}, for every $A\in\mathbb{R}_{\mathrm{sym}}^{d\times d}$,
\[(T_{\mu}-B_{\mu}\otimes B_{\mu})(A)=(T_{\mu}-\Sigma\otimes\Sigma)(A)=2\Sigma A\Sigma,\]
and so
\[\|(T_{\mu}-B_{\mu}\otimes B_{\mu})(A)\|_{\mathrm{F}}\leq 2\|\Sigma\|_{\mathrm{op}}^{2}\|A\|_{\mathrm{F}}.\]
Hence, $\|T_{\mu}-B_{\mu}\otimes B_{\mu}\|_{\mathrm{op}}\leq 2\|\Sigma\|_{\mathrm{op}}^{2}$, so by Proposition \ref{usefulsmubound}, it follows that $s(\mu)\leq\sqrt{2}\|\Sigma\|_{\mathrm{op}}$. This proves the upper bound for $s(\mu)$.

To prove the lower bound for $s(\mu)$, since the quantity $s(\mu)$ is orthogonally invariant in $\mu$, by a change of basis, we may assume that $\Sigma$ is a diagonal matrix with descending entries $\lambda_{1}\geq\ldots\geq\lambda_{d}$. Then $\mu=\mathcal{N}(0,\Sigma)$ coincides with the product measure $\mathcal{N}(0,\lambda_{1})\times\ldots\times\mathcal{N}(0,\lambda_{d})$ of one-dimensional normal distributions. Hence by Lemma \ref{smuproduct} and Lemma \ref{smu1dsinglenormal}, we have $s(\mu)\geq s(\mathcal{N}(0,\lambda_{1}))\geq 0.8\lambda_{1}=0.8\|\Sigma\|_{\mathrm{op}}$. This proves the lower bound for $s(\mu)$.
\end{proof}

\begin{proof}[Proof of the bounds for $s(\mu)$ in Example \ref{introgeneralmixture}]
We first prove the upper bound for $s(\mu)$. Let $\mu_{1}=\mathcal{N}(0,\Sigma_{1})$ and $\mu_{2}=\mathcal{N}(0,\Sigma_{2})$. Since $\mu=\frac{1}{2}\mu_{1}+\frac{1}{2}\mu_{2}$, we have $T_{\mu}=\frac{1}{2}T_{\mu_{1}}+\frac{1}{2}T_{\mu_{2}}$. So by Lemma \ref{tmugaussian}, for every $A\in\mathbb{R}_{\mathrm{sym}}^{d\times d}$,
\[T_{\mu}(A)=\frac{1}{2}T_{\mu_{1}}(A)+\frac{1}{2}T_{\mu_{2}}(A)=
\Sigma_{1}A\Sigma_{1}+\Sigma_{2}A\Sigma_{2}+\frac{1}{2}\langle A,\Sigma_{1}\rangle\Sigma_{1}+\frac{1}{2}\langle A,\Sigma_{2}\rangle\Sigma_{2}.\]
Since $B_{\mu}=\frac{1}{2}B_{\mu_{1}}+\frac{1}{2}B_{\mu_{2}}=\frac{1}{2}\Sigma_{1}+\frac{1}{2}\Sigma_{2}$, it follows that for every $A\in\mathbb{R}_{\mathrm{sym}}^{d\times d}$,
\begin{align*}
&(T_{\mu}-B_{\mu}\otimes B_{\mu})(A)\\=&
T_{\mu}(A)-\langle A,B_{\mu}\rangle B_{\mu}\\=&
\Sigma_{1}A\Sigma_{1}+\Sigma_{2}A\Sigma_{2}+\frac{1}{2}\langle A,\Sigma_{1}\rangle\Sigma_{1}+\frac{1}{2}\langle A,\Sigma_{2}\rangle\Sigma_{2}-\left\langle A,\frac{\Sigma_{1}+\Sigma_{2}}{2}\right\rangle\frac{\Sigma_{1}+\Sigma_{2}}{2}\\=&
\Sigma_{1}A\Sigma_{1}+\Sigma_{2}A\Sigma_{2}+\left\langle A,\frac{\Sigma_{1}-\Sigma_{2}}{2}\right\rangle\frac{\Sigma_{1}-\Sigma_{2}}{2},
\end{align*}
and thus,
\[\|(T_{\mu}-B_{\mu}\otimes B_{\mu})(A)\|_{\mathrm{F}}\leq\|\Sigma_{1}\|_{\mathrm{op}}^{2}\|A\|_{\mathrm{F}}+\|\Sigma_{2}\|_{\mathrm{op}}^{2}\|A\|_{\mathrm{F}}+\|A\|_{\mathrm{F}}\left\|\frac{\Sigma_{1}-\Sigma_{2}}{2}\right\|_{\mathrm{F}}^{2}.\]
Hence,
\[\|T_{\mu}-B_{\mu}\otimes B_{\mu}\|_{\mathrm{op}}\leq\|\Sigma_{1}\|_{\mathrm{op}}^{2}+\|\Sigma_{2}\|_{\mathrm{op}}^{2}+\left\|\frac{\Sigma_{1}-\Sigma_{2}}{2}\right\|_{\mathrm{F}}^{2},\]
and so by Proposition \ref{usefulsmubound},
\[s(\mu)\leq\|\Sigma_{1}\|_{\mathrm{op}}+\|\Sigma_{2}\|_{\mathrm{op}}+\frac{1}{2}\|\Sigma_{1}-\Sigma_{2}\|_{\mathrm{F}}.\]
This proves the upper bound for $s(\mu)$. To prove the lower bound for $s(\mu)$, by Lemma \ref{smumixture}, we have $s(\mu)\geq\frac{1}{2}s(\mathcal{N}(0,\Sigma_{1}))$. But by Example \ref{introsinglenormal}, $s(\mathcal{N}(0,\Sigma_{1}))\geq0.8\|\Sigma_{1}\|_{\mathrm{op}}$. Therefore, $s(\mu)\geq 0.4\|\Sigma_{1}\|_{\mathrm{op}}$. Interchanging the roles of $\Sigma_{1}$ and $\Sigma_{2}$, we also have $s(\mu)\geq 0.4\|\Sigma_{2}\|_{\mathrm{op}}$. Finally, since $\mu=\frac{1}{2}\mathcal{N}(0,\Sigma_{1})+\frac{1}{2}\mathcal{N}(0,\Sigma_{2})$, by the definition of $s(\mu)$, we have $s(\mu)\geq\frac{1}{2}\|\Sigma_{1}-\Sigma_{2}\|_{\mathrm{F}}$. We conclude that
\[s(\mu)\geq\max\left(0.4\|\Sigma_{1}\|_{\mathrm{op}},\,0.4\|\Sigma_{2}\|_{\mathrm{op}},\,\frac{1}{2}\|\Sigma_{1}-\Sigma_{2}\|_{\mathrm{F}}\right),\]
as desired.
\end{proof}

\section{Proof of the examples II}\label{examplesection}
In this section, we prove Example \ref{examplelambdastandardnormal}-Example \ref{examplelambdamixture}.
\begin{proof}[Proof of Example \ref{examplelambdastandardnormal}]
By Lemma \ref{tmuiid},
\[T_{\mu}(A)=2A+\mathrm{Tr}(A)I,\quad A\in\mathbb{R}_{\mathrm{sym}}^{d\times d}.\]
Let $\mathcal{H}_{1}$ be the span of $I$. Let $\mathcal{H}_{2}=\{A\in\mathbb{R}_{\mathrm{sym}}^{d\times d}|\,\langle A,I\rangle=0\}$. Then $\mathbb{R}_{\mathrm{sym}}^{d\times d}$ can be decomposed as the direct sum of the two mutually orthogonal subspaces $\mathcal{H}_{1}$ and $\mathcal{H}_{2}$.

We have $T_{\mu}(A)=(d+2)A$ for all $A\in\mathcal{H}_{1}$, whereas $T_{\mu}(A)=2A$ for all $A\in\mathcal{H}_{2}$. Therefore, the eigenvalues of $T_{\mu}$ are $d+2$ and $2$. The multiplicity of the eigenvalue $d+2$ is equal to $\mathrm{dim}\,\mathcal{H}_{1}=1$, whereas the multiplicity of the eigenvalue $2$ is equal to $\mathrm{dim}\,\mathcal{H}_{2}=\frac{d(d+1)}{2}-1$.
\end{proof}

\begin{proof}[Proof of Example \ref{examplelambdaiid}]
By Lemma \ref{tmuiid},
\[T_{\mu}(A)=2A+\mathrm{Tr}(A)I+(\mathbb{E}X_{1}^{4}-3)\mathrm{diag}(A),\quad A\in\mathbb{R}_{\mathrm{sym}}^{d\times d}.\]
Let $\mathcal{H}_{1}$ be the span of $I$. Let $\mathcal{H}_{2}$ be the set of all $d\times d$ diagonal matrices $A$ such that $\langle A,I\rangle=0$. Let $\mathcal{H}_{3}$ be the set of all $A\in\mathbb{R}_{\mathrm{sym}}^{d\times d}$ such that all the diagonal entries of $A$ are $0$. Then $\mathbb{R}_{\mathrm{sym}}^{d\times d}$ can be decomposed as the direct sum of the three mutually orthogonal subspaces $\mathcal{H}_{1},\mathcal{H}_{2},\mathcal{H}_{3}$. We have
\[T_{\mu}(A)=(d+\mathbb{E}X_{1}^{4}-1)A,\quad A\in\mathcal{H}_{1},\]
\[T_{\mu}(A)=(\mathbb{E}X_{1}^{4}-1)A,\quad A\in\mathcal{H}_{2},\]
\[T_{\mu}(A)=2A,\quad A\in\mathcal{H}_{3}.\]
Moreover, $\mathrm{dim}\,\mathcal{H}_{1}=1$, $\mathrm{dim}\,\mathcal{H}_{2}=d-1$, $\mathrm{dim}\,\mathcal{H}_{3}=\frac{d(d-1)}{2}$. Therefore, the eigenvalues of $T_{\mu}$ are
\[\underbrace{d+\mathbb{E}X_{1}^{4}-1}_{1},\;\underbrace{\mathbb{E}X_{1}^{4}-1,\ldots,\mathbb{E}X_{1}^{4}-1}_{d-1},\;\underbrace{2,\ldots,2}_{\frac{d(d-1)}{2}},\]
as announced.
\end{proof}

\begin{proof}[Proof of Example \ref{examplelambdamixture}]
For each $i\in\{1,\ldots,r\}$, let $\mu_{i}=\mathcal{N}(0,P_{i})$. Then $\mu=\frac{1}{r}\mu_{1}+\ldots+\frac{1}{r}\mu_{r}$ and so $T_{\mu}=\frac{1}{r}T_{\mu_{1}}+\ldots+\frac{1}{r}T_{\mu_{r}}$. By Lemma \ref{tmugaussian},
\[T_{\mu_{i}}(A)=2P_{i}AP_{i}+\langle A,P_{i}\rangle P_{i},\]
for all $A\in\mathbb{R}_{\mathrm{sym}}^{d\times d}$ and $i\in\{1,\ldots,r\}$. So
\begin{equation}\label{examplelambdamixtureproofeq1}
T_{\mu}(A)=\frac{1}{r}\sum_{i=1}^{r}(2P_{i}AP_{i}+\langle A,P_{i}\rangle P_{i}),
\end{equation}
for all $A\in\mathbb{R}_{\mathrm{sym}}^{d\times d}$. Let $\mathcal{H}_{1}$ be the span of $P_{1},\ldots,P_{r}$,
\[\mathcal{H}_{2}=\{A\in\mathbb{R}_{\mathrm{sym}}^{d\times d}|\; A=\sum_{i=1}^{r}P_{i}AP_{i}\quad\text{and}\quad\langle A,P_{i}\rangle=0\;\forall i\},\]
and
\[\mathcal{H}_{3}=\{A\in\mathbb{R}_{\mathrm{sym}}^{d\times d}|\,P_{i}AP_{i}=0\quad\forall i\}.\]
Then $\mathbb{R}_{\mathrm{sym}}^{d\times d}$ can be decomposed as the direct sum of the three mutually orthogonal subspaces $\mathcal{H}_{1},\mathcal{H}_{2},\mathcal{H}_{3}$. Indeed, it is clear that these 3 spaces are mutually orthogonal. To see that these 3 spaces span the entire $\mathbb{R}_{\mathrm{sym}}^{d\times d}$, observe that the span of $\mathcal{H}_{1}$ and $\mathcal{H}_{2}$ is equal to
\begin{equation}\label{examplelambdamixtureproofeq2}
\{A\in\mathbb{R}_{\mathrm{sym}}^{d\times d}|\,A=\sum_{i=1}^{r}P_{i}AP_{i}\}.
\end{equation}
This is because for every $A\in\mathbb{R}_{\mathrm{sym}}^{d\times d}$ and every $j\in\{1,\ldots,r\}$, we can write $P_{j}AP_{j}=\alpha_{j}P_{j}+(P_{j}AP_{j}-\alpha_{j}P_{j})$, where $\alpha_{j}\in\mathbb{R}$ is so that $\langle P_{j}AP_{j}-\alpha_{j}P_{j},P_{j}\rangle=0$. Then $\alpha_{j}P_{j}\in\mathcal{H}_{1}$ and $P_{j}AP_{j}-\alpha_{j}P_{j}\in\mathcal{H}_{2}$, and so $P_{j}AP_{j}$ is in the span of $\mathcal{H}_{1}$ and $\mathcal{H}_{2}$. Thus, if $A\in\mathbb{R}_{\mathrm{sym}}^{d\times d}$ satisfies $A=\sum_{j=1}^{r}P_{j}AP_{j}$, then $A$ is in the span of $\mathcal{H}_{1}$ and $\mathcal{H}_{2}$. This proves that the span of $\mathcal{H}_{1}$ and $\mathcal{H}_{2}$ is equal to the space in (\ref{examplelambdamixtureproofeq2}).

For every $A\in\mathbb{R}_{\mathrm{sym}}^{d\times d}$, we can write
\[A=\sum_{j=1}^{r}P_{j}AP_{j}+\left[A-\sum_{j=1}^{r}P_{j}AP_{j}\right].\]
Note that $\sum_{j=1}^{r}P_{j}AP_{j}$ is in the span of $\mathcal{H}_{1}$ and $\mathcal{H}_{2}$ by (\ref{examplelambdamixtureproofeq2}), whereas $A-\sum_{j=1}^{r}P_{j}AP_{j}$ is in $\mathcal{H}_{3}$. Therefore, $A$ is in the span of $\mathcal{H}_{1},\mathcal{H}_{2},\mathcal{H}_{3}$, and so $\mathcal{H}_{1},\mathcal{H}_{2},\mathcal{H}_{3}$ span the entire $\mathbb{R}_{\mathrm{sym}}^{d\times d}$.

By (\ref{examplelambdamixtureproofeq1}), since $\mathrm{Tr}(P_{j})=\frac{d}{r}$ for each $j\in\{1,\ldots,r\}$, we have
\[T_{\mu}(A)=\frac{1}{r}\left(2+\frac{d}{r}\right)A=\frac{d+2r}{r^{2}}A,\quad A\in\mathcal{H}_{1},\]
\[T_{\mu}(A)=\frac{2}{r}A,\quad A\in\mathcal{H}_{2},\]
\[T_{\mu}(A)=0,\quad A\in\mathcal{H}_{3}.\]
So the eigenvalues of $T_{\mu}$ are $\frac{d+2r}{r^{2}}$, $\frac{2}{r}$ and $0$. The multiplicity of the eigenvalue $\frac{d+2r}{r^{2}}$ is equal to $\mathrm{dim}\,\mathcal{H}_{1}=r$. The multiplicity of the eigenvalue $\frac{2}{r}$ is equal to $\mathrm{dim}\,\mathcal{H}_{2}$. To find $\mathrm{dim}\,\mathcal{H}_{2}$, note that since the span of $\mathcal{H}_{1}$ and $\mathcal{H}_{2}$ is equal to the space in (\ref{examplelambdamixtureproofeq2}), which has dimension $r\cdot\mathrm{dim}(\mathbb{R}_{\mathrm{sym}}^{\frac{d}{r}\times\frac{d}{r}})=r\cdot\frac{\frac{d}{r}(\frac{d}{r}+1)}{2}=\frac{d(d+r)}{2r}$, the dimension of $\mathcal{H}_{2}$ is $\frac{d(d+r)}{2r}-r$.
\end{proof}

\section{Unequal weight mixtures}\label{unequalsection}
\begin{proof}[Proof of Lemma \ref{unequal}]
Let $\Lambda$ be the set of all signed measures of the form $\mu_{1}-\mu_{2}$ for some $\mu_{1},\mu_{2}\in\mathcal{P}_{2}(\mathbb{R}^{d})$. Define $\Phi:\Lambda\to\mathbb{R}$ by
\[\Phi(\zeta)=\left\|\int_{\mathbb{R}^{d}}xx^{T}\,d\zeta(x)\right\|_{\mathrm{F}},\quad\zeta\in\Lambda.\]
We need to show that
\begin{align}\label{unequalproofeq1}
\frac{1}{2(1-\alpha)}\sup_{\mu=\frac{1}{2}\mu_{1}+\frac{1}{2}\mu_{2}}\Phi(\mu_{1}-\mu_{2})\leq&\sup_{\mu=\alpha\nu_{1}+(1-\alpha)\nu_{2}}\Phi(\nu_{1}-\nu_{2})\\\leq&
\frac{1}{2\alpha}\sup_{\mu=\frac{1}{2}\mu_{1}+\frac{1}{2}\mu_{2}}\Phi(\mu_{1}-\mu_{2}),\nonumber
\end{align}
where the middle is supremum is over all probability measures $\nu_{1},\nu_{2}$ on $\mathbb{R}^{d}$ such that $\mu=\alpha\nu_{1}+(1-\alpha)\nu_{2}$, and the other two suprema are over all probability measures $\mu_{1},\mu_{2}$ on $\mathbb{R}^{d}$ such that $\mu=\frac{1}{2}\mu_{1}+\frac{1}{2}\mu_{2}$.

Suppose that $\mu_{1},\mu_{2}$ satisfies $\mu=\frac{1}{2}\mu_{1}+\frac{1}{2}\mu_{2}$. Then we can write $\mu$ as the following weighted mixture
\[\mu=\alpha\mu_{1}+(1-\alpha)\left(\frac{\frac{1}{2}-\alpha}{1-\alpha}\mu_{1}+\frac{1}{2(1-\alpha)}\mu_{2}\right),\]
and
\[\mu_{1}-\left(\frac{\frac{1}{2}-\alpha}{1-\alpha}\mu_{1}+\frac{1}{2(1-\alpha)}\mu_{2}\right)=\frac{1}{2(1-\alpha)}(\mu_{1}-\mu_{2}).\]
Hence,
\[\sup_{\mu=\alpha\nu_{1}+(1-\alpha)\nu_{2}}\Phi(\nu_{1}-\nu_{2})\geq\Phi\left(\frac{1}{2(1-\alpha)}(\mu_{1}-\mu_{2})\right)=\frac{1}{2(1-\alpha)}\Phi(\mu_{1}-\mu_{2}).\]
This proves the first inequality in (\ref{unequalproofeq1}). To prove the second inequality in (\ref{unequalproofeq1}), suppose that $\nu_{1},\nu_{2}$ satisfies $\mu=\alpha\nu_{1}+(1-\alpha)\nu_{2}$. Then we can write $\mu$ as the following equal weight mixture 
\[\mu=\frac{1}{2}\left(2\alpha\nu_{1}+(1-2\alpha)\nu_{2}\right)+\frac{1}{2}\nu_{2},\]
and
\[\left(2\alpha\nu_{1}+(1-2\alpha)\nu_{2}\right)-\nu_{2}=2\alpha(\nu_{1}-\nu_{2}).\]
Hence,
\[\sup_{\mu=\frac{1}{2}\mu_{1}+\frac{1}{2}\mu_{2}}\Phi(\mu_{1}-\mu_{2})\geq\Phi\left(2\alpha(\nu_{1}-\nu_{2})\right)=2\alpha\Phi(\nu_{1}-\nu_{2}).\]
This proves the second inequality in (\ref{unequalproofeq1}).
\end{proof}

\section{Future directions}\label{discuss}
\subsection{More than two components}
Theorem \ref{thm:spectral_decomposition} and Corollary \ref{asymptoticresolution} show how $\lambda_{1}(T_{\mu})$ and $\lambda_{2}(T_{\mu})$ can be used to test the mixture model hypothesis for a given probability measure $\mu\in\mathcal{P}_{8}(\mathbb{R}^{d})$. In particular, if $\lambda_{2}(T_{\mu})\sim\lambda_{1}(T_{\mu})$, i.e., the two quantities are of the same order, then $s(\mu)\sim\|B_{\mu}\|_{\mathrm{F}}$, i.e., $\mu$ satisfies the mixture model hypothesis (where we assume $\mu$ satisfies the $L^{8}$-$L^{2}$ equivalence). As mentioned in Remark \ref{spectralgraphresemble}, this resembles the classical fact in spectral graph theory that a regular graph is connected if and only if the first largest eigenvalue of its adjacency matrix is strictly larger than the second largest eigenvalue. It is also known that a regular graph has at least $r\in\mathbb{N}$ disconnected components if and only if the first $r$ largest eigenvalues of its adjacency matrix coincide. 
Should this graph theoretic fact also have an analog here?
It would suggest that if $\lambda_{3}(T_{\mu})\sim\lambda_{1}(T_{\mu})$ then the measure $\mu$ has at least 3 components. 

While the mixture model hypothesis is a rigorous way to define what it means for a given measure $\mu$ to have at least 2 components, it is not clear what a good definition of having ``at least 3 components" should be in the non-parametric setting.  Indeed, if $\mu$ can be decomposed as a mixture $\mu=\frac{1}{2}\mu_{1}+\frac{1}{2}\mu_{2}$ of two probability measures $\mu_{1},\mu_{2}\in\mathcal{P}_{2}(\mathbb{R}^{d})$ with $\|B_{\mu_{1}}-B_{\mu_{2}}\|_{\mathrm{F}}>\epsilon\|B_{\mu}\|_{\mathrm{F}}$, where $\epsilon>0$ is a fixed threshold, then $\mu$ can also be decomposed as the following mixture of three probability measures:
\[\mu=\frac{1}{3}\mu_{1}+\frac{1}{3}\mu_{2}+\frac{1}{3}\mu.\]
Moreover, $\|B_{\mu_{1}}-B_{\mu}\|_{\mathrm{F}}=\|B_{\mu_{1}}-(\frac{1}{2}B_{\mu_{1}}+\frac{1}{2}B_{\mu_{2}})\|_{\mathrm{F}}=\frac{1}{2}\|B_{\mu_{1}}-B_{\mu_{2}}\|_{\mathrm{F}}>\frac{\epsilon}{2}\|B_{\mu}\|_{\mathrm{F}}$, and similarly we also have $\|B_{\mu_{2}}-B_{\mu}\|_{\mathrm{F}}>\frac{\epsilon}{2}\|B_{\mu}\|_{\mathrm{F}}$. In short, if it is only required for the components to have second order statistics matrices that are significantly different from each other, then the measure $\mu$ having at least 2 components automatically implies having at least 3 components. The authors believe that if $\lambda_{3}(T_{\mu})\sim\lambda_{1}(T_{\mu})$, then $\mu$ admits a mixture decomposition $\mu=\frac{1}{3}\mu_{1}+\frac{1}{3}\mu_{2}+\frac{1}{3}\mu_{3}$ in which $B_{\mu_{1}},B_{\mu_{2}},B_{\mu_{3}}$ satisfy a stronger ``separation property" (e.g., maybe linear independence in a quantitative sense), rather than just being significantly different from each other in Frobenius norm.
\subsection{Tensorization}
The Cheeger constant of a given probability measure $\mu$ on $\mathbb{R}^{d}$ quantifies its ``metric disconnectedness."  More precisely, it captures the extent to which the population represented by $\mu$ can be partitioned into two subpopulations that are mostly separated in the metric sense. On the other hand, the quantity $s(\mu)$ quantifies the ``statistical disconnectedness" of $\mu$, or more precisely, the extent to which the population represented by $\mu$ can be partitioned into two subpopulations that are statistically very different. It is known from \cite{BobkovHoudre} that the Cheeger constant (sometimes also called the isoperimetric constant) behaves well under tensorization. Thus, one may ask: does the same hold true for $s(\mu)$?  The answer is no, but for a trivial reason.   Namely, if $d=1$ and $\mu$ is uniformly distributed on the two points $-1,1$, then $s(\mu)=0$, while $s(\mu^{\otimes 2})\neq 0$. This is because $(-1)^{2}=1^{2}=1$. But Example \ref{introproduct} suggests that one can still control the quantity $s(\mu)$ when $\mu$ gets tensorized.  We propose the following.
\begin{problem}\label{tensorizeproblem}
For a given $\mu\in\mathcal{P}_{2}(\mathbb{R}^{d})$, find a sharp estimate for $\displaystyle\sup_{n\in\mathbb{N}}s(\mu^{\otimes n})$ up to a universal constant factor, where $\mu^{\otimes n}$ is the product measure $\underbrace{\mu\times\ldots\times\mu}_{n}$ on $\mathbb{R}^{dn}$.
\end{problem}

\medskip
\medskip

\noindent{\textbf{Acknowledgments:} The authors are grateful to Christian Houdr\'e for some useful discussions and for asking about Problem \ref{tensorizeproblem}.  J.K. acknowledges partial support from NSF DMS 2309782, DE SC0025312, and the Sloan Foundation.

\appendix\section{The $L^{p}$-$L^{2}$ equivalence assumption}\label{lpl2section}
The main results of this paper Theorem \ref{firstmain} and Theorem \ref{thm:spectral_decomposition} have the assumption of $L^{p}$-$L^{2}$ equivalence for some $p\geq 2$:
\begin{equation}\label{lpl2intro}
\left(\int_{\mathbb{R}^d}|\langle x,v\rangle|^{p}\,d\mu(x)\right)^{1/p}\leq\beta\left(\int_{\mathbb{R}^d}|\langle x,v\rangle|^{2}\,d\mu(x)\right)^{{1/2}},\quad v\in\mathbb{R}^{d},
\end{equation}
where $\beta\geq 1$ is a fixed parameter.
This is, in fact, a commonly used assumption in high dimensional probability (see, e.g., \cite{AZ, MP, SV, Tikhomirov}). In this appendix, we review some basic properties of this assumption and review some examples.

The condition (\ref{lpl2intro}) is invariant under pushforward of $\mu$ by any linear transformation on $\mathbb{R}^{d}$. Moreover, if $\mu_{1},\ldots,\mu_{r}$ are probability measures satisfying (\ref{lpl2intro}) and $\alpha_{1},\ldots,\alpha_{r}>0$ with $\sum_{i=1}^{r}\alpha_{i}=1$, then the mixture $\sum_{i=1}^{r}\alpha_{i}\mu_{i}$ also satisfies (\ref{lpl2intro}), except with $\beta$ being replaced by $\displaystyle\beta\cdot\max_{1\leq i\leq r}\alpha_{i}^{(1/p)-(1/2)}$, since
\begin{align}\label{mixturelpl2}
&\left(\int_{\mathbb{R}^d}|\langle x,v\rangle|^{p}\,d\left(\sum_{i=1}^{r}\alpha_{i}\mu_{i}\right)(x)\right)^{2/p}\\\leq&
\sum_{i=1}^{r}\left(\alpha_{i}\int_{\mathbb{R}^d}|\langle x,v\rangle|^{p}\,d\mu_{i}(x)\right)^{2/p}\nonumber\\\leq&
\sum_{i=1}^{r}\alpha_{i}^{2/p}\cdot\beta^{2}\int_{\mathbb{R}^d}|\langle x,v\rangle|^{2}\,d\mu_{i}(x)\nonumber\\\leq&
\left(\beta^{2}\max_{1\leq i\leq r}\alpha_{i}^{(2/p)-1}\right)\sum_{i=1}^{r}\alpha_{i}\int_{\mathbb{R}^d}|\langle x,v\rangle|^{2}\,d\mu_{i}(x).\nonumber
\end{align}
for all $v\in\mathbb{R}^{d}$.

The condition (\ref{lpl2intro}) covers a wide range of probability measures $\mu$ on $\mathbb{R}^{d}$ for which we can control the $\beta$ so that it is independent of the dimension $d$. For $v\in\mathbb{R}^{d}$ and $p\geq 1$, define $\|v\|_{p}:=(\sum_{i=1}^{d}|\langle v,e_{i}\rangle|^{p})^{1/p}$.
\begin{example}\label{examplelpl2gaussian}
If $\mu=\mathcal{N}(w,B)$ where $w\in\mathbb{R}^{d}$ and $B$ is a $d\times d$ positive semidefinite matrix, then $\mu$ satisfies (\ref{lpl2intro}) with $\beta=C\sqrt{p}$ for some absolute constant $C$. Indeed, if $X$ is a random vector in $\mathbb{R}^{d}$ with distribution $\mu$, then for every $v\in\mathbb{R}^{d}$, the random variable $\langle X,v\rangle$ has normal distribution with mean $\langle w,v\rangle$ and standard deviation $\|B^{1/2}v\|_{2}$. Hence, $(\int_{\mathbb{R}^{d}}|\langle x,v\rangle|^{2}\,d\mu(x))^{1/2}=(\mathbb{E}|\langle X,v\rangle|^{2})^{1/2}=(|\langle w,v\rangle|^{2}+\|B^{1/2}v\|_{2}^{2})^{1/2}$, and $(\int_{\mathbb{R}^{d}}|\langle x,v\rangle|^{p}\,d\mu(x))^{1/p}=(\mathbb{E}|\langle X,v\rangle|^{p})^{1/p}=(\mathbb{E}|\langle w,v\rangle+\|B^{1/2}v\|_{2}g_{0}|^{p})^{1/p}\leq|\langle w,v\rangle|+\|B^{1/2}v\|_{2}(\mathbb{E}|g_{0}|^{p})^{1/p}$ by Minkowski's inequality, where $g_{0}$ is a standard normal random variable in $\mathbb{R}$. Since $(\mathbb{E}|g_{0}|^{p})^{1/p}\leq C_{0}\sqrt{p}$ for some absolute constant $C_{0}$, we conclude that $\mu$ satisfies (\ref{lpl2intro}) with $\beta=C_{0}\sqrt{p}+1$.
\end{example}

\begin{example}\label{examplelpl2gaussianmixture}
If $\mu=\sum_{i=1}^{r}\alpha_{i}\mathcal{N}(w_{i},B_{i})$ is a mixture of normal distributions, where $w_{i}\in\mathbb{R}^{d}$, $B_{i}$ is a $d\times d$ positive semidefinite matrix, $\alpha_{i}>0$, for all $1\leq i\leq r$, and $\sum_{i=1}^{r}\alpha_{i}=1$, then $\mu$ satisfies (\ref{lpl2intro}) with $\displaystyle\beta=C\sqrt{p}\cdot\max_{1\leq i\leq r}\alpha_{i}^{(1/p)-(1/2)}$ for some absolute constant $C$. This follows from Example \ref{examplelpl2gaussian} and (\ref{mixturelpl2}).
\end{example}

\begin{example}
Suppose that $X_{1},\ldots,X_{d}$ are independent, identically distributed random variables with $\mathbb{E}X_{1}=0$, $\mathbb{E}X_{1}^{2}=1$ and $\mathbb{E}|X_{1}|^{p}<\infty$. Let $\mu$ be the distribution of the random vector $X=(X_{1},\ldots,X_{d})^{T}$ in $\mathbb{R}^{d}$. Then for every $v\in\mathbb{R}^{d}$, we have $(\int_{\mathbb{R}^{d}}|\langle x,v\rangle|^{p}\,d\mu(x))^{1/p}=(\mathbb{E}|\sum_{i=1}^{d}\langle v,e_{i}\rangle X_{i}|^{2})^{1/2}=\|v\|_{2}$
and by Rosenthal's inequality \cite[Theorem 3]{Rosenthal}, we have
\begin{eqnarray*}
\left(\int_{\mathbb{R}^{d}}|\langle x,v\rangle|^{p}\,d\mu(x)\right)^{1/p}&=&\left(\mathbb{E}\left|\sum_{i=1}^{d}\langle v,e_{i}\rangle X_{i}\right|^{p}\right)^{1/p}\\&\leq&
K_{p}\max\{(\mathbb{E}|X_{1}|^{p})^{1/p}|\|v\|_{p},\|v\|_{2}\}\leq K_{p}(\mathbb{E}|X_{1}|^{p})^{1/p}\|v\|_{2},  
\end{eqnarray*}
where $K_{p}>0$ is a constant that depends only on $p$. Therefore, $\mu$ satisfies (\ref{lpl2intro}) with $\beta=K_{p}(\mathbb{E}|X_{1}|^{p})^{1/p}$.
\end{example}

\begin{example}
If $X$ is an isotropic subgaussian random vector in $\mathbb{R}^{d}$ and $\mu$ is the distribution of $X$, then $\mu$ satisfies (\ref{lpl2intro}) with $\beta=C\|X\|_{\psi_{2}}\sqrt{p}$ for some absolute constant $C$, where $\|X\|_{\psi_{2}}$ is the subgaussian norm of $X$. See Proposition 2.6.1, Definition 2.6.4, Definition 3.2.5, Definition 3.4.1 in \cite{Romanbook}.
\end{example}

\begin{example}\label{examplelpl2mixtureconvex}
If $\mu$ is the uniform distribution on a convex body $K$ in $\mathbb{R}^{d}$, then $\mu$ satisfies (\ref{lpl2intro}) with $\beta=Cp$ for some absolute constant $C$. See Subsection 3.4.3 and Proposition 2.8.1 in \cite{Romanbook} for the case when $K$ is isotropic. However, since (\ref{lpl2intro}) is invariant under pushforward of $\mu$ by any linear transformation on $\mathbb{R}^{d}$, it holds even when $K$ is not isotropic.
\end{example}

On the other hand, below is an example of $\mu$ where the the smallest $\beta$ satisfying (\ref{lpl2intro}) grows as the dimension $d$ gets large. 
\begin{example}
Suppose that $\mu$ is the uniform distribution on the $2d$ points $\pm e_{1},\ldots,\pm e_{d}$. Then for all $v\in\mathbb{R}^{d}$, we have $(\int_{\mathbb{R}^{d}}|\langle x,v\rangle|^{4}\,d\mu(x))^{1/4}=(\frac{1}{d}\sum_{i=1}^{d}|\langle v,e_{i}\rangle|^{4})^{1/4}=d^{-1/4}\|v\|_{4}$ and $(\int_{\mathbb{R}^{d}}|\langle x,v\rangle|^{2})^{1/2}=(\frac{1}{d}\sum_{i=1}^{d}|\langle v,e_{i}\rangle|^{2})^{1/2}=d^{-1/2}\|v\|_{2}$. Hence, when $p=4$, the smallest $\beta$ satisfying (\ref{lpl2intro}) is $\beta=d^{1/4}$.
\end{example}
\end{document}